\documentclass[notitlepage,11pt,reqno]{amsart}
\usepackage{amssymb,amsmath,amscd,xspace,amsthm,tocvsec2}
\usepackage[mathscr]{eucal}
\usepackage[breaklinks=true]{hyperref}

\newcommand{\N}{\ensuremath{\mathbb N}}
\newcommand{\Z}{\ensuremath{\mathbb Z}}
\newcommand{\R}{\ensuremath{\mathbb R}}
\newcommand{\C}{\ensuremath{\mathbb C}}
\newcommand{\bbh}{\ensuremath{\mathbb H}}

\newcommand{\cald}{\ensuremath{\mathcal D}}

\newcommand{\calo}{\ensuremath{\mathcal O}}
\newcommand{\calp}{\ensuremath{\mathcal P}}
\newcommand{\scrc}{\ensuremath{\mathscr C}}
\newcommand{\scrl}{\ensuremath{\mathscr L}}
\newcommand{\scrr}{\ensuremath{\mathscr R}}
\newcommand{\ghat}{\ensuremath{\widehat{\mathfrak{g}}}}
\newcommand{\qhat}{\ensuremath{\widehat{\mathfrak{q}}}}
\newcommand{\pihat}{\ensuremath{\widehat{\pi}}}
\newcommand{\bfl}{\ensuremath{\boldsymbol{\lambda}}}
\newcommand{\bfk}{\ensuremath{{\bf K}}}
\newcommand{\h}{\ensuremath{{\bf h}}}
\newcommand{\n}{\ensuremath{{\bf n}}}
\newcommand{\I}{\ensuremath{{\bf I}}}

\newcommand{\vg}{\ensuremath{V_{\overline{\mathfrak{g}}}}}

\newcommand{\lie}[1]{\ensuremath{\mathfrak{#1}}}
\newcommand{\one}[1]{\ensuremath{{#1}_{(1)}}}
\newcommand{\two}[1]{\ensuremath{{#1}_{(2)}}}
\newcommand{\scrlug}[1]{\ensuremath{\mathscr{L}^{#1}(\mathfrak{Ug})}}
\newcommand{\comment}[1]{}

\DeclareMathOperator{\wt}{\ensuremath{{wt}\xspace}}
\DeclareMathOperator{\conv}{\ensuremath{{conv}\xspace}}
\DeclareMathOperator{\Lie}{\ensuremath{{Lie}\xspace}}
\DeclareMathOperator{\chr}{\ensuremath{{char}\xspace}}
\DeclareMathOperator{\Sym}{\ensuremath{{Sym}\xspace}}
\DeclareMathOperator{\Hom}{\ensuremath{{Hom}\xspace}}
\DeclareMathOperator{\End}{\ensuremath{{End}\xspace}}
\DeclareMathOperator{\Ann}{\ensuremath{{Ann}\xspace}}
\DeclareMathOperator{\sn}{\ensuremath{{sn}\xspace}}
\DeclareMathOperator{\ad}{\ensuremath{{ad}\xspace}}
\DeclareMathOperator{\Sh}{\ensuremath{{Sh}\xspace}}

\newtheorem{theorem}[equation]{Theorem}
\newtheorem{prop}[equation]{Proposition}
\newtheorem{cor}[equation]{Corollary}

\theoremstyle{definition}
\newtheorem{remark}[equation]{Remark}
\newtheorem{defn}[equation]{Definition}
\newtheorem{ex}[equation]{Example}

\numberwithin{equation}{section}

\usepackage[letterpaper]{geometry}
\begin{document}
\title{Representations of complex semi-simple Lie groups and Lie
algebras}
\author[Apoorva Khare]{Apoorva Khare$^\dagger$}\thanks{$\dagger$Departments of Mathematics and Statistics, Stanford
University, Stanford, CA - 94305, USA}
\thanks{Email: {\tt khare@stanford.edu}}
\thanks{This work was supported in part by DARPA Grant  \# YFA
N66001-11-1-4131.}
\subjclass[2010]{Primary: 22E46; Secondary: 17B20}
\date{\today}

\begin{abstract}
This article is an exposition of the 1967 paper by Parthasarathy, Ranga
Rao, and Varadarajan, on irreducible admissible Harish-Chandra modules
over complex semisimple Lie groups and Lie algebras. It
was written in Winter 2012 to be part of a special collection organized
to mark 10 years and 25 volumes of the series {\em Texts and Readings
in Mathematics} (TRIM). Each article in this collection
is intended to give nonspecialists in its field, an appreciation for the
impact and contributions of the paper being surveyed. Thus, the author
has kept the prerequisites for this article down to a basic course on
complex semisimple Lie algebras.

While it arose out of the grand program of Harish-Chandra on admissible
representations of semisimple Lie groups, the work by Parthasarathy et al
also provided several new insights on highest weight modules and related
areas, and these results have been the subject of extensive research over
the last four decades. Thus, we also discuss its results and follow-up
works on the classification of irreducible Harish-Chandra modules; on the
PRV conjecture and tensor product multiplicities; and on PRV determinants
for (quantized) affine and semisimple Lie algebras.
\end{abstract}
\maketitle

\settocdepth{section}
{\tt \tableofcontents}

Lie groups and Lie algebras occupy a prominent and central place in
mathematics, connecting differential geometry, representation theory,
algebraic geometry, number theory, and theoretical physics. In some
sense, the heart of (classical) representation theory is in the study of
the semisimple Lie groups. Their study is simultaneously simple in its
beauty, as well as complex in its richness. From Killing, Cartan, and
Weyl, to Dynkin, Harish-Chandra, Bruhat, Kostant, and Serre, many
mathematicians in the twentieth century have worked on building up the
theory of semisimple Lie algebras and their universal enveloping
algebras. Books by Borel, Bourbaki, Bump, Chevalley, Humphreys, Jacobson,
Varadarajan, Vogan, and others form the texts for (introductory) graduate
courses on the subject.

The purpose of this article is to provide an exposition of the famous
1967 paper \cite{PRV2} by Parthasarathy, Ranga Rao, and Varadarajan on a
class of irreducible Banach space representations of a complex semisimple
Lie group. This paper was written in a period containing some of the
other classic works in the subject: Harish-Chandra's pioneering work on
the principal series representations, and his results on the annihilators
of simple modules and central characters; Kostant's work on harmonic
polynomials and on his character formula; and papers about Steinberg's
formula and Verma modules, to name a few.

In this article, we attempt to explain some of the key ideas and main
results of \cite{PRV2}. Given the wide variety of new concepts proposed,
as well as its impact on subsequent research in the field, the paper
ranks alongside these other works mentioned above.

\subsection{}

The basic motivation for the paper \cite{PRV2} arose out of the important
works \cite{Har2,Har3} of Harish-Chandra, in which he constructed a large
family of infinite-dimensional irreducible representations of a real
semisimple Lie group $G$. Harish-Chandra generalized the constructions by
Gelfand and Naimark in the case when $G$ is complex semisimple, and his
work is regarded today as a cornerstone in the field. For instance, he
showed how irreducible representations are subquotients of the principal
series representations.

In their work, which followed a few years after \cite{Har2,Har3},
Parthasarathy et al returned to the simpler setting of complex semisimple
Lie groups, where they were able to use Harish-Chandra's results to
obtain a deeper understanding of the structure of Harish-Chandra's Banach
space representations of $G$. Their paper develops a beautiful theory of
such representations, each of which decomposes into finite-dimensional
modules when restricted to the maximal compact subgroup of $G$. The
authors go on to develop the theory of minimal types, and refine
Harish-Chandra's methods (in the complex case) for classifying such
irreducible Banach space representations.

\subsection{}

Although this is the primary motivation for \cite{PRV2}, the paper
develops and proves many other results that have since influenced and
inspired a large body of research in the field. We mention a few of these
here (and elaborate upon them in future sections). First, the authors
provided a multiplicity formula for the classical ``tensor product
decomposition" problem: given two simple finite-dimensional modules over
a complex semisimple Lie algebra, can one write down the decomposition of
their tensor product? Combinatorial results due to Kostant, Sternberg,
and Brauer were known at the time; however, they required double
summations over the Weyl group, computing the Kostant partition function,
and cancelling terms in the summation, which made them increasingly
harder to implement.

In \cite{PRV2}, the authors proposed a formula which was somewhat
simpler, directly involving the tensor factors in question. This formula
has since been widely used in the literature (as we point out in this
article), including in the setting of quantum affine algebras and
symmetrizable Kac-Moody algebras, as well as current algebras and other
semidirect product Lie algebras.

\subsection{}

Next, a byproduct of this ``PRV Theorem" (or formula) was that every such
tensor product contains a unique ``largest" summand (the ``Cartan
component"), and a unique ``smallest" summand (the ``PRV component", or
``minimal type"). The former was well-known to be the sum of the two
highest weights in question, but the latter was new. Subsequently, the
authors and Kostant conjectured the existence of other components, the
so-called ``generalized PRV components". These are simple modules that
occur as direct summands of the tensor product, and their (dominant
integral) highest weights are Weyl group-linear combinations of the
highest weights of the two tensor factors.

This ``PRV Conjecture" has since been proved using multiple techniques in
the semisimple as well as Kac-Moody settings. Moreover, it has inspired
subsequent research that has led to many contributions in understanding
the original problem of computing tensor product multiplicities. Once
again, we will discuss these facts in detail below. Throughout this
article, we will discuss the results in \cite{PRV2} in the special case
of $\lie{sl}_2(\C)$, in order to provide a working example - one which we
hope will give the reader a greater feel for the results being stated.

\subsection{}

We end this introduction with one last application. In \cite{PRV2}, using
Kostant's ``separation of variables" theorem, the authors defined a set
of matrices indexed by pairs of dominant integral weights, whose entries
are polynomials on the Cartan subalgebra. The determinants of these
matrices yield information about the annihilators of Verma modules, and
of their simple quotients. (This is related to the Shapovalov form.)

These ``PRV determinants" have since been widely studied, not just in the
semisimple case, but in the quantum (affine) and the super-reductive
settings as well. In these settings, PRV determinants can be used to
determine whether or not the annihilators of Verma modules are generated
by their intersection with the center.

\section{Notation and preliminaries}

We assume for the purposes of this article that the reader is familiar
with basic results concerning the structure of complex semisimple Lie
algebras; see \cite{H}, for instance. We now set some basic notation,
which also serves as a quick summary of the theory. Given a complex
semisimple Lie algebra $\lie{g}$, fix a Cartan subalgebra $\lie{h}
\subset \lie{g}$ (which is abelian and self-normalizing in $\lie{g}$).
Then $\lie{g}$ has a direct sum decomposition: $\lie{g} = \lie{h} \oplus
\bigoplus_{\alpha \in R} \lie{g}_\alpha$, where $R$ is the set of roots
and $\lie{g}_\alpha$ is the one-dimensional root space for each $\alpha
\in R \subset \lie{h}^*$.
Here, an element $\lambda \in \lie{h}^*$ is called a weight (in
\cite{PRV2} it is called a ``rank"), and if $M$ is an $\lie{h}$-module,
then its $\lambda$-weight space is defined to be:
\[ M_\lambda := \{ m \in M : h \cdot m = \lambda(h) m\ \forall h \in
\lie{h} \}. \]

\noindent The weights of $M$, denoted $\wt(M)$, are those $\lambda \in
\lie{h}^*$ for which $M_\lambda \neq 0$. $M$ is a weight module if $M =
\bigoplus_{\lambda \in \lie{h}^*} M_\lambda$.
For instance, $\lie{g}$ is a $\lie{g}$-module under the adjoint action,
and a weight $\lie{h}$-module. The nonzero weights of $\lie{g}$ are
precisely the roots: $\wt(\lie{g}) = R \coprod \{ 0 \}$.

A simple example to keep in mind is $\lie{g} = \lie{sl}_2(\C)$. This has
a basis $\{ e, f, h \}$ with defining relations:
\[ [h,e] = 2e, \qquad [h,f] = -2f, \qquad [e,f] = h. \]

\noindent Here, $\lie{h} = \C \cdot h$ and $R = \{ \pm \alpha \}$, where
$\alpha(h) = 2$. Thus, $\lie{g}_\alpha = \C \cdot e, \lie{g}_{-\alpha} =
\C \cdot f$.

\subsection{Weights and lattices}

Now let $W$ be the associated Weyl group and $(-,-)$ the Killing form for
$\lie{g}$. Then $(-,-)$ induces an isomorphism $: \lie{h} \to \lie{h}^*$.
Fix a positive system $R^+ \subset R$ of roots - or equivalently, the
subset $\Pi = \{ \alpha_i : i \in I \}$ of simple roots indexed by a set
$I$. Then $\Pi$ is a basis of $\lie{h}^*$, and $R = R^+ \coprod R^-$,
where $R^+ = -R^- = R \cap \Z_{\geq 0} \Pi \subset \lie{h}^*$. For each
$i \in I$, suppose $h'_i \longleftrightarrow \alpha_i$ via the Killing
form; now define the co-roots to be $h_i := (2 / \alpha_i(h'_i)) \cdot
h'_i \in \lie{h}$.

Next, choose Chevalley generators $e_i \in \lie{g}_{\alpha_i}$ and $f_i
\in \lie{g}_{-\alpha_i}$ that generate (as above) a copy of $\lie{sl}_2$
together with $h_i$. Then the $e_i$ and $f_i$ generate nilpotent
subalgebras $\lie{n}^\pm$ of $\lie{g}$, and the corresponding Borel
(maximal solvable) subalgebras are: $\lie{b}^\pm := \lie{h} \oplus
\lie{n}^\pm$.

We now come to distinguished lattices inside the set of weights. A weight
$\lambda \in \lie{h}^*$ is said to be dominant if $\lambda(h_i) \geq 0$
for all $i \in I$, and integral if $\lambda(h_i) \in \Z$.
The set of integral weights is a lattice $\Lambda \subset \lie{h}^*$,
whose $\Z$-basis is the set of fundamental weights $\{ \varpi_i : i \in I
\}$. They are defined by: $\varpi_i(h_j) := \delta_{ij}$. Let $\Lambda^+$
denote the set of dominant integral weights - which are simply $\Z_{\geq
0}$-linear combinations of the $\varpi_i$. Then $\Lambda^+$ is also
(in bijection with) the set of dominant characters of a maximal torus $T$
of $G$ (where $G$ is a complex connected Lie group such that $\lie{g} =
\Lie(G)$).

The weight lattice $\Lambda$ also contains the root lattice $\Z \Pi$ with
$\Z$-basis $\Pi$. The group $W$ acts on $\lie{h}^*$ and preserves either
lattice. It is generated by the simple reflections $\{ s_i : i \in I \}$
which act via: $s_i(\lambda) := \lambda - 2 \lambda(h_i) \alpha_i$. The
reflections $s_i$ satisfy the Coxeter relations according to the Dynkin
diagram of $\lie{g}$, and $W$ is a finite group with a well-defined
length-function $\ell : W \to \Z_{\geq 0}$, the associated Bruhat order,
and a unique longest element $w_\circ = w_\circ^{-1}$.

For example, for $\lie{g} = \lie{sl}_2(\C)$, $\Pi = R^+ = \{ \alpha \}$,
where $\alpha(h) = 1$. The associated fundamental weight is $\varpi =
\frac{1}{2} \alpha$, so that $\Z \Pi = 2 \Lambda$ and $\Lambda^+ =
\Z_{\geq 0} \varpi$. Moreover, $W = S_2 = \{ 1, s = w_\circ \}$, where
$w_\circ \lambda = - \lambda$ for all weights $\lambda \in \lie{h}^* = \C
\varpi = \C \alpha$.

\subsection{Finite-dimensional representations}

A representation or module of a group $G$ is simply a group homomorphism
$\varpi : G \to GL(V)$ for some (real or complex) vector space $V$.
Similarly, a $\lie{g}$-module $V$ is a Lie algebra homomorphism $\varpi :
\lie{g} \to \End(V) = \lie{gl}(V)$. We say that $\varpi$ is irreducible
if $V$ has no nonzero proper submodule; completely reducible or
semisimple if $V$ is a direct sum of irreducible submodules; and
finite-dimensional if $\dim V < \infty$.

When $\lie{g}$ is semisimple, the irreducible finite-dimensional
representations are all weight modules for $\lie{h}$, and parametrized by
$\Lambda^+$. Here is a quick construction: suppose $\lie{Ug}$ is the
universal enveloping algebra of $\lie{g}$, and for any $\lambda \in
\lie{h}^*$, let $I_\lambda$ be the left $\lie{Ug}$-ideal generated by
$\ker \lambda \subset \lie{h}$ and $\{ e_i : i \in I \}$.
The Verma module $M(\lambda)$ is defined to be the quotient $\lie{Ug} /
I_\lambda$. $M(\lambda)$ is a weight module and $\wt(M) = \lambda -
\Z_{\geq 0} \Pi$. Moreover, a cyclic generator of $M(\lambda)$ lies in
$M(\lambda)_\lambda = \C \cdot \overline{1_{\lie{Ug}}}$, which is called
the ``highest weight space".

The modules $M(\lambda)$ were studied by Verma in his thesis and in
\cite{Ver}; they are of tremendous importance in representation theory -
not only for semisimple Lie algebras, but also Kac-Moody and Virasoro
algebras, quantum groups, and other algebras with triangular
decomposition. Every Verma module $M(\lambda)$ has a largest maximal
submodule and hence a unique simple quotient; denote this by
$V(\lambda)$. Then $V(\lambda)$ also has the same properties as
$M(\lambda)$ (mentioned in the previous paragraph); moreover, the modules
$V(\lambda)$ are pairwise non-isomorphic for $\lambda \in \lie{h}^*$.

Note that $V(\lambda)$ is finite-dimensional if and only if $\lambda \in
\Lambda^+$, and these exhaust all finite-dimensional simple
$\lie{g}$-modules up to isomorphism. The dual space to a $\lie{g}$-module
is also a $\lie{g}$-module; for instance, $V(\lambda)^* \cong V(-w_\circ
\lambda)$ if $\lambda \in \Lambda^+$. Moreover, every finite-dimensional
$\lie{g}$-module is semisimple; in other words, every indecomposable
finite-dimensional $\lie{g}$-module is irreducible.

For example, if $\lie{g} = \lie{sl}_2$, then for every $0 \leq n \in \Z$,
there exists a unique irreducible $\lie{sl}_2$-module $V(n) \cong V(n)^*$
of dimension $n+1$. (Note that we are abusing notation by using $V(n)$ to
refer to $V(n \varpi)$.) $V(n)$ contains a vector $v_n$ of weight $n$ (a
``highest weight vector"), and a basis $v_{n-2i} := (f^i / i!) v_n$ of
weight vectors, for $0 \leq i < \dim V(n)$. One checks that for all $i$,
\begin{equation}\label{Esl2}
h \cdot v_{n-2i} := (n-2i) v_{n-2i}, \quad e \cdot v_{n-2i} = (n-i+1)
v_{n-2i+2}, \quad f \cdot v_{n-2i} = (i+1) v_{n-2i-2},
\end{equation}

\noindent where we set $v_{n+2} = v_{-n-2} = 0$. A concrete example of
$V(n)$ is provided by the space of homogeneous polynomials in $X,Y$ of
total degree $n$. Define
\[ P_n := \ker \left(-n + X \frac{\partial}{\partial X} + Y
\frac{\partial}{\partial Y} \right) \subset \C[X,Y]. \]

\noindent Now define $\rho_n : \lie{sl}_2 \to \End_\C(P_n)$ via:
\[ \rho_n(e) := X \frac{\partial}{\partial Y}, \qquad \rho_n(f) := Y
\frac{\partial}{\partial X}, \qquad \rho_n(h) := X
\frac{\partial}{\partial X} - Y \frac{\partial}{\partial Y}. \]

\noindent Then $P_n \cong V(n)$ as $\lie{sl}_2$-modules.

\subsection{Central characters}

Given an associative algebra $A$, denote its center by $Z(A)$.
An important tool in studying Verma modules and finite-dimensional
modules over a (complex) semisimple Lie algebra $\lie{g}$ is the center
$Z(\lie{Ug})$. Classical results of Chevalley and Harish-Chandra imply
that this is a polynomial algebra in $|I|$ (algebraically independent)
generators. Moreover, for all $\lambda \in \lie{h}^*$, there exists a
central character (i.e., an algebra homomorphism) $\chi(\lambda) :
Z(\lie{Ug}) \to \C$, such that $\ker \chi(\lambda)$ kills $M(\lambda)$
and hence $V(\lambda)$ for all $\lambda \in \lie{h}^*$. In particular,
every $z \in Z(\lie{Ug})$ acts on $V(\lambda)$ by a scalar (for each
$\lambda$).

The following important results due to Harish-Chandra completely classify
and explain better, the set of central characters. (These are also known
as {\it infinitesimal characters} in the literature.) To state the
results, we need some notation: define $\rho := \frac{1}{2} \sum_{\alpha
\in R^+} \alpha \in \lie{h}^*$; then $\rho = \sum_{i \in I} \varpi_i \in
\Lambda^+$ and $w_\circ \rho = - \rho$. Now define the twisted action of
$W$ on $\lie{h}^*$ via:
\[ w * \lambda := w(\lambda + \rho) - \rho, \qquad \forall w \in W,
\lambda \in \lie{h}^*. \]

\noindent Then $w * -$ induces an algebra automorphism of $\Sym \lie{h} =
P(\lie{h}^*)$ (the space of complex polynomials on $\lie{h}^*$) for all
$w \in W$. Moreover, given any $w \in W$, define $\lie{n}^\pm_w :=
\bigoplus_{\alpha \in R \cap \Z_{\geq 0} (w \Pi)} \lie{g}_{\pm \alpha}$.
Then $\lie{g} = \lie{n}^-_w \oplus \lie{h} \oplus \lie{n}^+_w$ for all $w
\in W$; for example, when $w=1$, this decomposition is precisely $\lie{g}
= \lie{n}^- \oplus \lie{h} \oplus \lie{n}^+$. Hence $(\lie{Ug})_0 \subset
\lie{Uh} \oplus \lie{n}^-_w (\lie{Ug}) \lie{n}^+_w$ by the
Poincar\'e-Birkhoff-Witt theorem. Define the Harish-Chandra map $\beta^{w
\Pi}$ to be the projection $: (\lie{Ug})_0 \twoheadrightarrow \lie{Uh} =
\Sym \lie{h}$.

\begin{theorem}[\cite{Har1}]\label{Thc}
For all $w \in W$, $\beta^{w\Pi}$ is a ring homomorphism $: (\lie{Ug})_0
\twoheadrightarrow \Sym \lie{h}$, which restricts to a ring isomorphism
$: Z(\lie{Ug}) \to (\Sym \lie{h})^{(W,*)}$. Moreover, for all $\lambda
\in \lie{h}^*$, $\chi(\lambda) = \lambda \circ \beta^\Pi$. (Here,
$\lambda$ extends to an algebra map on $\Sym \lie{h}$.)
Every character of $Z(\lie{Ug})$ equals $\chi(\lambda)$ for some $\lambda
\in \lie{h}^*$. Moreover, $\chi(\lambda) = \chi(\mu) \Leftrightarrow
\lambda = w * \mu$ for some $w \in W$.
\end{theorem}

\noindent For instance, when $\lie{g} = \lie{sl}_2(\C)$, $|I| = 1$ and
\[ Z(\lie{U}(\lie{sl}_2(\C))) = \C[\Delta], \qquad \Delta = 4 fe + h^2 +
2h, \qquad \beta^\Pi(\Delta) = h^2 + 2h. \]

\noindent Then for all $z \in \C$, the Casimir element $\Delta$ acts on
$V(z \varpi)$ as the scalar $\chi(z \varpi)(\Delta) = z^2 + 2z$. Note
that $\rho = \varpi$ and $\chi(z \varpi) \equiv \chi(z' \varpi)$ on
$Z(\lie{Ug}) = \C[\Delta]$, if and only if $z+z'=-2$ - i.e., $z' \varpi =
s * (z \varpi)$. Similarly, $s * h = -h-2$, so:
\[ \beta^\Pi(\Delta) = (h+1)^2 - 1 = ((s*h)+1)^2 - 1 = s *
\beta^\Pi(\Delta). \]

\section{Harish-Chandra modules}

We start our discussion of \cite{PRV2} with the main motivation: the
works of Harish-Chandra. The representations studied by Parthasarathy,
Ranga Rao, and Varadarajan are known today as (irreducible) {\it
admissible Harish-Chandra modules}. They were first studied in the
setting of real semisimple Lie groups by Harish-Chandra in
\cite{Har2,Har3}.

\subsection{}

For the better part of a century, and since the advent of quantum
mechanics, mathematicians have been interested in unitary representations
and harmonic analysis of locally compact (abelian) topological groups.
One of the basic results in this direction is the Peter-Weyl Theorem,
which says that every unitary (Hilbert space) representation of a compact
group decomposes as a direct sum of finite-dimensional irreducible
submodules. Given the correspondence between complex semisimple groups
and compact groups (discovered by Weyl), the class of unitary
representations of complex semisimple Lie groups $G$ and their maximal
compact subgroups $K$ has been a subject of wide interest and research in
the literature.

To explain the motivation for \cite{PRV2}, some notation is now needed.
Given $G \supset K$ as above, $\lie{k} = \Lie(K)$ is the compact form of
$\lie{g} = \Lie(G)$. Let $\lie{g}^\C := \Lie(G) \otimes_\R \C$ be the
complexification of its Lie algebra; this is a complex semisimple Lie
algebra that contains the reductive subalgebra $\lie{k}^\C := \Lie(K)
\otimes_\R \C$.
In his works cited above, Harish-Chandra initiated the study of a class
of irreducible infinite-dimensional $G$-modules that was larger than the
class of unitary $G$-modules (yet these modules were direct sums of
finite-dimensional $K$-modules). Harish-Chandra constructed and studied
these modules algebraically, via their correspondence to
$\lie{g}^\C$-modules (when $G$ has finite center). This correspondence
was known for finite-dimensional modules, but he showed how to extend it
to a deep and powerful theory of Banach-space representations of $G$.

More precisely, given a continuous Banach space $G$-representation $\pi$,
whose restriction to $K$ contains every irreducible $K$-module with at
most finite multiplicity (this is called ``admissibility"),
Harish-Chandra considered its subspace of {\it K-finite vectors} (i.e.,
the vectors that lie in finite-dimensional $K$-stable subspaces) - or
more precisely, the $\lie{k}^\C$-finite vectors. This subspace is called
the {\it infinitesimal representation associated to $\pi$}. Two such
Banach space representations are said to be {\it infinitesimally
equivalent} if their infinitesimal representations are equivalent. (For
instance, Harish-Chandra showed in \cite{Har2} that two irreducible
unitary $G$-modules are equivalent if and only if they are
infinitesimally equivalent.) One of the crown jewels of his work is the
{\it subquotient theorem} \cite{Har3}, which says that every such
admissible $V$ is infinitesimally equivalent to a subquotient of a
Hilbert space representation of $G$ (the ``principal series
representations").

\subsection{}

We now return to \cite{PRV2}, where the authors are interested in using
Harish-Chandra's work to gain a deeper understanding of a special case of
this situation: namely, when $G$ is already a complex group. (This
setting was also studied earlier - by Weyl in relating complex and
compact groups, but also by Gelfand and Naimark \cite{GN}, and
Harish-Chandra as well.)
By \cite{Har3}, it turns out that every $\lie{k}^\C$-finite irreducible
$G$-representation $V$ (with at most finite multiplicities) has an
infinitesimal character. In other words, the center
$Z(\lie{U}(\lie{g}^\C))$ acts by scalars on it.
As noted in \cite{PRV2,Var}, the problem of describing the infinitesimal
equivalence classes of irreducible Banach space $G$-representations
(which are ``admissible", hence equipped with an infinitesimal character)
can now be reduced by Harish-Chandra's work, to describing the
irreducible $\lie{k}^\C$-finite representations - i.e., the so-called
simple ``$(\lie{g}^\C, \lie{k}^\C)$-modules" (or ``$(\lie{g}^\C,
K)$-modules"). Later in this section, we will mention certain prominent
features from Harish-Chandra's approach in the real semisimple case, as
we specialize them to the complex case in \cite{PRV2}.

Thus, the primary motivation in \cite{PRV2} was to study the irreducible
$G$-representations when $G$ is a complex semisimple group, by applying
the methods and deep results from \cite{Har2,Har3}.
For instance, the authors are able to simplify Harish-Chandra's
description of the closed subspaces of the principal series
representations, which yield Banach (actually, Hilbert) space
$G$-representations. Furthermore, the theory of minimal types developed
in \cite{PRV2} helps obtain a deeper understanding of these simple
$(\lie{g}^\C, \lie{k}^\C)$-modules.

We now introduce the setting of \cite{PRV2}. If $G$ is a complex Lie
group, then $\lie{g} := \Lie(G)$ is a complex Lie algebra, and $\Lie(K)$
is its compact (real) form. Thus as real Lie algebras, $\lie{g} = \Lie(K)
\oplus \sqrt{-1} \cdot \Lie(K)$ in the complex structure of $\lie{g}$.
The complexified pair $(\lie{g}^\C, \lie{k}^\C)$ is isomorphic to
$(\lie{g} \times \lie{g}, \overline{\lie{g}})$, where
$\overline{\lie{g}}$ is the diagonal copy of $\lie{g}$ embedded in
$\lie{g} \times \lie{g}$.\footnote{This is achieved using a conjugation
$X \mapsto X^c$ of $\lie{g}$ that fixes $\lie{k}$. Thus, $\lie{g}$ embeds
inside $\lie{g}^\C$ via: $X \mapsto (X^c, X)$, and when restricted to
$\lie{k}$, one obtains: $X \mapsto (X,X)$ - whence we get that
$\lie{k}^\C = \overline{\lie{g}}$.}
Now the $\lie{g} \times \lie{g}$-modules of interest (studied by
Harish-Chandra in general) are the ones that decompose into direct sums
of finite-dimensional $\overline{\lie{g}}$-modules with at most finite
multiplicities.

Here is the precise framework studied in \cite{PRV2} (and henceforth in
this article), stated here in a slightly more general setting.

\begin{defn}
Suppose $\lie{g}$ is a complex reductive finite-dimensional Lie algebra
contained in a complex Lie algebra $\ghat$. Define the category
$\scrc(\ghat, \lie{g})$ to be the full subcategory of $\ghat$-modules,
such that every object is isomorphic to a direct sum of
finite-dimensional irreducible $\lie{g}$-modules $\cald$, each of which
occurs with finite multiplicity. (This last condition is termed {\it
$\lie{g}$-admissibility}.) Define $[V : \cald]$ to be this multiplicity
(which may be zero if no summand is isomorphic to $\cald$); this integer
does not depend on the direct sum decomposition of $V$. (Note that we
assume $\cald \neq 0$.)

If $V$ is in $\scrc(\ghat, \lie{g})$ and $\cald$ is a (nonzero) simple
$\lie{g}$-module, the {\it isotypical subspace} $V_\cald$ of $V$ is
defined as the (finite-dimensional) span of all the $\lie{g}$-submodules
of $V$ that are isomorphic to $\cald$. Clearly, the center $Z(\lie{U}
\ghat)$ preserves $V_\cald$ for each $\cald$, and hence acts locally
finitely on $V$. Moreover, $[V : \cald]$ then equals $[V_\cald : \cald] =
\dim V_\cald / \dim \cald$.
\end{defn}

\subsection{Examples of Harish-Chandra modules in the literature}

The goal of \cite{PRV2} was to study the simple objects in the category
$\scrc(\lie{g} \times \lie{g}, \overline{\lie{g}})$. Before elaborating
on their results, we remark that various families of Harish-Chandra
modules have been widely studied in the literature.
For example, Harish-Chandra modules are examples of integrable
$\overline{\lie{g}}$-modules - i.e., $\overline{\lie{g}}$-modules where
every vector is contained in a finite-dimensional
$\overline{\lie{g}}$-module.
Here are some other examples; in them, we always assume that $\lie{g}$ is
semisimple (and complex).

\begin{ex}
Suppose $\lie{g}$ is semisimple and $\lie{g}_0$ is its compact real form.
Let $G_0$ be a compact Lie group with Haar measure $\mu$, such that
$\lie{g}_0 = \Lie(G_0)$. Then by the Peter-Weyl Theorem, a dense subspace
$V$ of $L^2(G_0,\C,\mu)$ is an object of $\scrc(\lie{g}, \lie{g})$.
Moreover, $[V : V(\lambda)] = \dim V(\lambda)$ for all $\lambda \in
\Lambda^+$.
\end{ex}

\begin{ex}
Recall that $\lie{Ug}$ is a direct sum of finite-dimensional
$\lie{g}$-modules, since every term in its standard filtration is. Is it
also an object of $\scrc(\lie{g}, \lie{g})$? The answer is no - in fact,
{\it no} finite-dimensional module occurs with finite nonzero
multiplicity. To see this, note by Kostant's ``separation of variables
theorem" \cite{Kos2} that $\lie{Ug}$ is free as a module (under
multiplication) over its center:
\begin{equation}\label{Ekostant}
\lie{Ug} \cong \bbh(\lie{g}) \otimes Z(\lie{Ug}),
\end{equation}

\noindent where $\bbh(\lie{g})$ is (isomorphic as a $\lie{g}$-module to)
the space of ``harmonic polynomials on $\lie{g}$", and is stable under
the adjoint action of $\lie{g}$ on $\lie{Ug}$. Thus, the multiplicity in
$\lie{Ug}$ of every finite-dimensional module is either $0$ or $\dim
Z(\lie{Ug})$, which is infinite. In particular, $\lie{Ug}$ is not
admissible.

However, $\bbh(\lie{g})$ is indeed an object in $\scrc(\lie{g},
\lie{g})$; in fact, Kostant proved in \cite{Kos2} that $[\bbh(\lie{g}) :
V(\lambda)] = \dim V(\lambda)_0$ for all $\lambda \in \Lambda^+$. This is
the starting point for another important contribution of \cite{PRV2} to
the theory of semisimple and affine (quantized) Lie algebras - the
so-called ``PRV determinants". We will discuss these in a later section.
\end{ex}

\begin{ex}
Simple finite-dimensional $\lie{g} \times \lie{g}$-modules are clearly in
$\scrc(\lie{g} \times \lie{g}, \overline{\lie{g}})$, by Weyl's Theorem of
complete reducibility. This example is also the starting point for a
result and a conjecture from \cite{PRV2} (the ``PRV Theorem" and ``PRV
conjecture"), that have since been extensively used and generalized in
the literature. We address these in detail in the next section.
\end{ex}

\begin{ex}
The above example of $\rho_n$ for $\lie{g} = \lie{sl}_2(\C)$ can be used
to produce an object in $\scrc(\lie{sl}_2(\C), \lie{sl}_2(\C))$ as
follows: the $\lie{sl}_2(\C)$-module 
\[ \C[X,Y] = \bigoplus_{n \geq 0} P_n = \bigoplus_{n \geq 0} V(n) \]

\noindent is clearly such an object.

Note that $P_n = V(n) = \Sym^{n-1} (V(1))$ for all $n \in \N$. Thus,
$\C[X,Y] = \Sym (V(1))$. With this in mind, we can generalize the above
example to $\lie{g} = \lie{sl}_n(\C)$, as it acts on its simple module
$\C^n$ (for $n>1$). Consider the modules $\Sym^k(\C^n) \subset
(\C^n)^{\otimes k}$ for $k \geq 0$. Identifying a basis of $\C^n$ with
commuting variables $X_1, \dots, X_n$, it is not hard to show that as
$\lie{g}$-modules, $\Sym^k(\C^n)$ is precisely the space $P_{n,k}$ of
homogeneous polynomials in $X_1, \dots, X_n$ of total degree $k$, where
$e_{ij}$ acts on $P_{n,k}$ as $X_i \partial_j$ for all $1 \leq i, j \leq
n$ and all $k$.

One can now check that $P_{n,k}$ is a simple module over
$\lie{sl}_n(\C)$\footnote{See {\tt
http://math.stackexchange.com/questions/120338} for the sketch of a
proof.}. Since $n>1$, hence
$\displaystyle \dim P_{n,k} = \binom{k+n-1}{n-1}$ is increasing in $k$.
Thus the $P_{n,k}$ are non-isomorphic for fixed $n$, and so
\[ \C[X_1, \dots, X_n] = \bigoplus_{k \geq 0} P_{n,k} = \Sym(\C^n) \]

\noindent is indeed an object in $\scrc(\lie{sl}_n(\C), \lie{sl}_n(\C))$.
\end{ex}

\begin{ex}
If $\lie{g}$ is semisimple and $\lie{h}$ is the Cartan subalgebra of
$\lie{g}$, then $\scrc(\lie{g}, \lie{h})$ is the category of (admissible)
weight modules. There has been extensive research on the study and
classification of irreducible (admissible) weight modules; see
\cite{Mat3} for more on this. We remark that Mathieu also studied
Harish-Chandra modules in other settings in \cite{Mat2}: the Virasoro
algebra, the Cartan algebra, and the affine Kac-Moody algebras (as
mentioned in the conclusion to {\em loc.~cit.}).

Moreover, a very special family of admissible weight modules constitutes
the objects of the Bernstein-Gelfand-Gelfand Category $\calo$, which was
introduced in \cite{BGG1}.
A lot of research has been undertaken on the Category $\calo$ in various
settings in modern representation theory - including semisimple and
Kac-Moody Lie algebras, the quantum groups associated with them, the
Virasoro algebra, and more modern constructions such as rational
Cherednik algebras, infinitesimal Hecke algebras, and $W$-algebras. In
particular, for semisimple $\lie{g}$, the classification of irreducible
admissible weight modules (by work of Mathieu \cite{Mat3} and others) as
well as of primitive ideals (by work of Duflo \cite{Du3} and others)
reduces to the study of simple objects in $\calo$. See
\cite{H2,Jo2,Kh3,MP} for additional references and results.
\end{ex}

\subsection{A key class of homomorphisms}

We now outline Harish-Chandra's strategy for studying admissible
irreducible $G$-representations, as it is used by Parthasarathy et al in
the complex setting.
Given $\lie{g} \subset \ghat$ as above, let $\Omega$ denote the
centralizer of $\lie{g}$ in $\lie{U} \ghat$. (This is denoted by
$\lie{O}$ in \cite{Var}.) Then $Z(\lie{Ug}) + Z(\lie{U} \ghat) \subset
\Omega \subset \lie{U} \ghat$ is a chain of algebras\footnote{In
Varadarajan's reminiscences \cite{Var}, he points out on Page (xii) that
$\Omega$ is highly nonabelian, so that the first inclusion is not an
equality in general.}.
Now suppose $V$ is an object of $\scrc(\ghat, \lie{g})$, and $\cald$ is a
simple finite-dimensional $\lie{g}$-module such that $[V : \cald] = r
>0$. Then $\cald \cong V(\lambda)$ as finite-dimensional (and hence,
highest-weight) $\lie{g}$-modules, and $V_\cald \cong V(\lambda) \otimes
\C^r$ as a $\lie{g}$-module (i.e., $\C^r$ is the multiplicity space).
Multiplication by every $z \in \Omega$ preserves the highest weight space
$(V_\cald)_\lambda = V(\lambda)_\lambda \otimes \C^r$; this yields a
representation $\eta_{V,\cald}$ of $\Omega$ into $\C^r$. (This is called
$\eta((\nu^0), \pi)$ in \cite{PRV2}, where $\pi = V$ and $(\nu^0) =
\cald$.) Moreover $V_\cald$ now decomposes as $\cald \otimes \C^r$, under
the joint action of $\lie{g}$ and $\Omega$.

As a special case, suppose $r=1$. Then $\eta_{V,\cald}$ is a homomorphism
$: \Omega \to \C$. These homomorphisms are the key tools used in
\cite{PRV2} to study simple admissible Harish-Chandra modules, as we now
explain.\medskip

Suppose $\ghat = \lie{g} \times \lie{g} \supset \overline{\lie{g}}$. In
order to study simple objects in $\scrc(\lie{g} \times \lie{g},
\overline{\lie{g}})$, the authors of \cite{PRV2} follow the approach
suggested by Harish-Chandra in \cite{Har3}: if $r = [V : \cald] > 0$,
then as above, $V_\cald \cong \cald \otimes \C^r$ under the joint action
of $\lie{U}\overline{\lie{g}}$ and $\Omega$ - and moreover, the
$\Omega$-representation $\eta_{V, \cald}$ is simple.
Now the following remarkable fact holds: {\it the equivalence class of
the representation $\eta_{V,\cald}$ of $\Omega$ determines that of $V$,
for every component $\cald$ with $r>0$}.
More precisely, if $V, V'$ are simple objects of $\scrc(\lie{g} \times
\lie{g}, \overline{\lie{g}})$ and $\cald$ is a simple finite-dimensional
$\overline{\lie{g}}$-module such that $[V : \cald] + [V' : \cald] > 0$,
then
\begin{equation}\label{Ehar}
[V : \cald] = [V', \cald] > 0, \ \eta_{V, \cald} \cong_\Omega \eta_{V',
\cald} \Longleftrightarrow V \cong V'.
\end{equation}

\noindent (See \cite{LMC} for a generalization of this fact.)
Thus, a ``first approach" would be to fix various $\cald$ and study the
$\Omega$-modules $\eta_{V,\cald}$ for all $V$ with $[V : \cald] > 0$. The
authors remark in \cite{PRV2} that such an approach was not very fruitful
and so a different method had to be adopted. Their contribution was to
introduce and study the following notion.

\begin{defn}
Suppose $V$ is a simple object in $\scrc(\lie{g} \times \lie{g},
\overline{\lie{g}})$ and $\lambda \in \Lambda^+$. We say that $\lambda$
(or $\vg(\lambda)$) is a {\it minimal type} of $V$ if $[V : \vg(\lambda)]
> 0$, and
\[ [V : \vg(\mu)] > 0 \implies \lambda \in \wt \vg(\mu). \]
\end{defn}

\noindent It is clear that there is at most one minimal type for each
$V$.

Now the strategy is as follows: first study a class of modules $V$ for
which the minimal type $\cald$ can be shown to exist. These are the
irreducible finite-dimensional representations of $\lie{g} \times
\lie{g}$, and it turns out that there is an explicit recipe to compute
the homomorphism $\eta_{V,\cald}$ in this case. This recipe involves
$\eta_{V,\cald}$ equalling a polynomial-valued homomorphism $\h^{\Pi'} :
\Omega \to P(\lie{h}^* \times \lie{h}^*)$, evaluated at some $\lambda \in
\Lambda^+, \nu \in \Lambda$ - in other words, $\h^{\Pi'}(-;\lambda,\nu) :
\Omega \to \C$. (This is explained in a later section.)

The authors then go on to explicitly construct a family $\displaystyle \{
\pihat_{\lambda, \nu} : \nu \in \Lambda, \lambda \in \lie{h}^* \}$ of
simple objects in $\scrc(\lie{g} \times \lie{g}, \overline{\lie{g}})$,
each of which has a minimal type $\overline{\nu}$ occurring with
multiplicity $r=1$. (This family necessarily includes the
finite-dimensional simple $\lie{g} \times \lie{g}$-modules, as we will
see below.) The authors show that for each such $\lambda$ and $\nu$, the
related ``key homomorphism" $\displaystyle \eta_{\pihat_{\lambda,\nu},
\overline{\nu}} : \Omega \to \C$ turns out to be the same recipe
$\h^{\Pi'}$ as above, now evaluated at (more general points) $\lambda,
\nu$. Thus, Equation \eqref{Ehar} can be applied to discuss the
classification of these modules $\pihat_{\lambda,\nu}$. (Here and
henceforth, we abuse notation and use $\eta_{\pihat_{\lambda,\nu},
\overline{\nu}}$ to refer to $\eta_{\pihat_{\lambda,\nu},
\vg(\overline{\nu})}$.)\medskip

Thus, the starting point for \cite{PRV2} - and even earlier, for
Varadarajan and Varadhan in 1963 for the special case of $\lie{g} =
\lie{sl}_n(\C)$ - was to prove the assertion that finite-dimensional
irreducible $\lie{g} \times \lie{g}$-modules have minimal types. Note
that such a simple module has highest weight in $\Lambda^+ \times
\Lambda^+ \subset (\lie{h} \times \lie{h})^*$, so we can write it as
$V(\lambda, \mu)$. It is clear that for all $X \in \lie{g}$, its image in
\[ \overline{\lie{g}} \subset \lie{g} \times \lie{g} \subset
\lie{U}(\lie{g} \times \lie{g}) = \lie{Ug} \otimes \lie{Ug} \]

\noindent is precisely $X \otimes 1 + 1 \otimes X$. Thus, restricting
$V(\lambda, \mu)$ to $\overline{\lie{g}}$ amounts to considering the
module $\vg(\lambda) \otimes \vg(\mu)$. In other words, the study of the
minimal type in this setting involves computing the direct summands of
the tensor product - i.e., computing Clebsch-Gordan coefficients. This
classical problem is the focus of the next section.

\subsection{Digression on minimal type due to Vogan}

We end this section with a few remarks on the notion of minimal type. The
more widely accepted notion of {\it minimal $K$-type} (or {\it lowest
$K$-type}) in the literature is due to Vogan \cite{Vo}, and differs from
the above notion (in \cite{PRV2}). More precisely, Vogan defines a weight
$\lambda \in \Lambda^+$ to be a {\it lowest $K$-type} for an admissible
Harish-Chandra $(G,K)$-module $M$, if:
\begin{itemize}
\item $V(\lambda)$ is a $K$-submodule of $M$; and
\item Among all $\mu$ such that $V(\mu)$ is a $K$-submodule of $M$, the
quantity $(\mu + 2 \rho, \mu + 2 \rho)$ is minimized at $\mu = \lambda$.
\end{itemize}

One can now ask what is the relation between these two notions. Note that
the definition due to Vogan guarantees existence of the minimal type (for
irreducible admissible representations), but not uniqueness. In fact,
uniqueness does not hold when $G$ is a (general) real group, such as
$SL(2,\R)$. However, uniqueness of the minimal $K$-type is guaranteed if
$G$ is a complex group; see \cite{Zh2} for more details.

On the other hand, the definition in \cite{PRV2} guarantees uniqueness
but not existence. If this version of the minimal type does exist, then
it is necessarily a minimal type due to Vogan. This can be shown using
the following generalization of \cite[Lemma 13.4.C]{H} (whose proof is
the same as that of {\it loc.~cit.}), with $A = \wt V(\mu)$ for any
finite-dimensional $K$-submodule $V(\mu)$ of $M$:

\begin{prop}
Suppose $A \subset \Lambda$ is $W$-stable, with highest weight $\mu$. (In
other words, $A \subset \mu - \Z_{\geq 0} \Pi$.) Fix $\lambda \in A$ and
$0 < c \in \R$. Then $\mu \in \Lambda^+$ and $A$ is finite; moreover,
\[  (\lambda + c \rho, \lambda + c \rho) \leq (\mu + c \rho, \mu + c
\rho), \]

\noindent with equality if and only if $\lambda = \mu$.
\end{prop}

\section{Tensor products, minimal types, and the (K)PRV conjecture}

In this section, we discuss \cite[Section 2.2]{PRV2}, which contains
several results, as well as a related conjecture, that have been
extremely influential on subsequent research in the field. These results
and the conjecture have to do with the classical question of computing
tensor product multiplicities.
Although the primary motivation of Parthasarathy et al was to study
tensor products in order to prove the existence and uniqueness of minimal
types, some of these statements have been subsequently generalized and
have contributed to several aspects of the multiplicity problem. We will
list some of the relevant papers and results presently.

\subsection{Tensor product multiplicities and the PRV Theorem}

We start by recalling the notion of {\it Littlewood-Richardson
coefficients}. By Weyl's theorem of complete reducibility, given
$\lambda, \mu \in \Lambda^+$, we can decompose
\[ V(\lambda) \otimes V(\mu) = \bigoplus_{\nu \in \Lambda^+}
m^\nu_{\lambda,\mu} V(\nu), \]

\noindent where the multiplicities $m^\nu_{\lambda,\mu} =
m^\nu_{\mu,\lambda}$ are the coefficients in question, also known as
tensor product multiplicities. (In the rest of this article, we will
abuse notation and denote $\vg(\lambda)$ by $V(\lambda)$.) If
$m^\nu_{\lambda, \mu} > 0$, we say that $V(\nu)$ is a {\it component} of
$V(\lambda) \otimes V(\mu)$. For example, $V(\lambda + \mu)$ is always a
component, generated by $V(\lambda)_\lambda \otimes V(\mu)_\mu$, and
$m^{\lambda+\mu}_{\lambda,\mu} = 1$ for all $\lambda,\mu \in \Lambda^+$.

The determination of the multiplicities $m^\nu_{\lambda,\mu}$ is a
longstanding open problem in the literature - as is the simpler problem
of computing whether or not $m^\nu_{\lambda, \mu}$ is positive. Efforts
to answer these questions have been ongoing since even before
\cite{PRV2}. For instance, in his famous paper \cite{Kos1}, Kostant
proved his multiplicity formula, and also showed a necessary condition
for $V(\nu)$ to be a component: it must be of the form $\nu = \lambda +
\mu_1 \in \Lambda^+$ for some $\mu_1 \in \wt(V(\mu))$.
Moreover, $m^\nu_{\lambda,\mu} \leq \dim V(\mu)_{\nu - \lambda}$. By work
\cite{St} of Steinberg (using Kostant's multiplicity formula), the
following was also known:
\[ m^\nu_{\lambda,\mu} =  \sum_{w \in W} \sn(w) \dim V(\mu)_{w * \nu -
\lambda} = \sum_{w,w' \in W} \sn(w) \sn(w') \calp(w'(\mu + \rho) - w(\nu
+ \rho) + \lambda). \]

\noindent Here, $\calp : \Lambda^+ \to \N$ is the Kostant partition
function (which is also defined to be zero on $\Lambda \setminus
\Lambda^+$), and $\sn : W \to \{ \pm 1 \}$ is the sign homomorphism,
which is $-1$ on all simple reflections $s_i$.
Steinberg's results imply \cite{Kum5} that if $(\lambda + \mu')(h_i) \geq
-1$ for all $\mu' \in \wt(V(\mu))$ and $i \in I$, then
$m^\nu_{\lambda,\mu} = \dim V(\mu)_{\nu - \lambda}$. Kostant had shown a
special case of this result in \cite{Kos1}, where he assumed that
$(\lambda + \mu')(h_i) \geq 0$.\medskip

In \cite{PRV2}, the following multiplicity formula is proved. Given $\mu,
\nu \in \Lambda^+$ and $\gamma \in \lie{h}^*$, define:
\begin{eqnarray*}
V^+(\mu; \gamma, \nu) & := & \{ v \in V(\mu)_\gamma : e_i^{\nu(h_i) + 1}
v = 0\ \forall i \in I \},\\
V^-(\mu; \gamma, \nu) & := & \{ v \in V(\mu)_\gamma : f_i^{\nu(h_i) + 1}
v = 0\ \forall i \in I \}.
\end{eqnarray*}

\begin{theorem}[\cite{PRV2}]\label{Tprv1}
For all $\lambda, \mu, \nu \in \Lambda^+$,
\[ m^\nu_{\lambda, \mu} = \dim V^+(\mu; \nu - \lambda, \lambda) = \dim
V^+(\nu; \lambda + w_\circ \mu, - w_\circ \mu). \]

\noindent Now given $\gamma \in \lie{h}^*$, $\dim V^+(\mu; \gamma, \nu) =
\dim V^-(\mu; w_\circ \gamma, -w_\circ \nu)$.
\end{theorem}

Here is a typical application of this result, which shows how to compute
multiplicities.

\begin{ex}\label{Eprv}
Suppose $\lambda,\nu \in \Lambda^+$. If $[V(\lambda) \otimes V(\lambda)^*
: V(\nu)] > 0$, then $\lambda - w_\circ \lambda - \nu \in \Z_{\geq 0}
\Pi$ by Kostant's results, which implies that $\nu \in \Lambda^+ \cap \Z
\Pi$. For every such $\nu$, Theorem \ref{Tprv1} now says:
\[ m^\nu_{\lambda, - w_\circ \lambda} = \dim V^+(\nu; 0, \lambda) := \dim
\{ v \in V(\nu)_0 : e_i^{\lambda(h_i) + 1} v = 0 \ \forall i \in I \}. \]

\noindent Thus if $\lambda(h_i)$ is large enough for all $i$ (e.g.,
$\lambda = n \rho$ for $n \gg 0$), then
\[ [V(\lambda) \otimes V(\lambda)^* : V(\nu)] = \dim V(\nu)_0 > 0, \qquad
\forall \nu \in \Lambda^+ \cap \Z \Pi. \]

\noindent where the last inequality follows by a result from \cite{Ha},
used below to prove Proposition \ref{Pmax}. \qed
\end{ex}

The advantage of the ``PRV Theorem" \ref{Tprv1} over some of the earlier
formulae in the literature is that it calculates the multiplicities
directly and without cancellation. For instance, note that the above
result of Steinberg involves a double summation over the Weyl group - and
cancellations of terms - and hence is not suitable for practical
computations. Several years prior to \cite{PRV2}, Brauer had proposed
another such formula in \cite{Br}; it is similar to a result of Klimyk in
\cite{Kl}, which appeared in the same year as \cite{PRV2}. The result is
stated as Exercise $24.9$ in \cite{H}\footnote{See also: {\tt
http://mathoverflow.net/questions/85593/}}. It computes the
multiplicities $m^\nu_{\lambda,\mu}$ as sums of dimensions $\dim
V(\mu)_{\mu'}$, but with coefficients of $\pm 1$ and $0$, which again
implies the need to perform cancellation (of formal characters).

It is mentioned in \cite{PRV2} that Kostant had obtained Theorem
\ref{Tprv1} previously but had not published it; for a historical account
of this result, see \cite{Kos4}.

The PRV Theorem \ref{Tprv1} has been widely used and generalized in the
literature. In \cite{CP}, Chari and Pressley extend a special case of
this result to show that simple integrable modules over affine Lie
algebras are quotients of tensor products. In \cite{Jo2}, Joseph studies
a similar result for a general symmetrizable (quantum) Kac-Moody Lie
algebra - as does Mathieu in \cite{Mat1}.
Among other applications, Young and Zegers start from Theorem \ref{Tprv1}
in \cite{YZ} and relate Dorey's rule to $q$-characters of fundamental
representations of quantum affine algebras of type ADE. Panyushev and
Yakimova study variants and consequences of the result in \cite{PY}.

From a personal viewpoint, the author has used the PRV Theorem in his
paper \cite{CKR} with Chari and Ridenour, to provide examples of families
of finite and infinite-dimensional Koszul algebras which naturally arise
out of module categories over semidirect products $\lie{g} \ltimes
V(\lambda)$. The result was also used by Chari and Greenstein
\cite{CG1,CG3} in the study of representations of the truncated current
algebra $\lie{g}[t] / (t^2)$, as well as in other works of Chari and her
collaborators, and of Greenstein. These papers have applications in the
study of Kirillov-Reshetikhin modules over quantum affine algebras.

\subsection{Minimal type}

We again start by considering the decomposition of the tensor product
into its simple module components. Consider the example where $\lie{g} =
\lie{sl}_2(\C)$ and $0 \leq \mu \leq \lambda \in \Z_{\geq 0}$. (Once
again, we abuse notation and use $\lambda \in \Z$ to refer to $\lambda
\cdot \varpi \in \Lambda$.) Recall the well-known {\it Clebsch-Gordan
formula} for $\lie{sl}_2(\C)$:
\begin{equation}\label{Eclebsch}
V(\lambda) \otimes V(\mu) = V(\lambda + \mu) \oplus V(\lambda + \mu - 2)
\oplus \dots \oplus V(\lambda - \mu).
\end{equation}

\noindent We see that there are two distinguished components in this
direct sum:
\begin{itemize}
\item The ``largest" component is $V(\lambda + \mu)$, in that every
highest weight occurring on the right, belongs to $\wt V(\lambda + \mu)$.
This is called the {\it Cartan component} or the {\it maximal type}, and
occurs with multiplicity $1$. It is generated by the one-dimensional
vector space $V(\lambda)_\lambda \otimes V(\mu)_\mu$, which is the tensor
product of the two highest weight spaces.

\item The ``smallest" component is $V(\lambda - \mu)$, in that $\lambda -
\mu$ is a weight of every summand occurring on the right. This is called
the {\it PRV component} (after the authors of \cite{PRV2}) or the {\it
minimal type}, and it also occurs with multiplicity $1$.
\end{itemize}

It is reasonable to ask if these results extend to all semisimple
$\lie{g}$. Remarkably, the authors of \cite{PRV2} found the answer of
this question to be positive! To understand it, one must first make sense
of what the minimal type is for general $\lie{g}$. Note above that we
could have interchanged $\lambda$ and $\mu$, since the tensor product is
``commutative" (i.e., the Hopf algebra $\lie{Ug}$ is cocommutative).
Thus, to choose the minimal type, one chooses the dominant integral
weight from among $\{ \lambda - \mu, \mu - \lambda \} = W(\lambda +
w_\circ \mu)$. Supporting evidence is obtained from Theorem \ref{Tprv1},
where if $\lambda + w_\circ \mu \in \Lambda^+$, then substituting it for
$\nu$ yields:
\begin{equation}\label{E43}
m^{\lambda + w_\circ \mu}_{\lambda,\mu} = \dim V^+(\lambda + w_\circ \mu;
\lambda + w_\circ \mu, -w_\circ \mu) = 1.
\end{equation}

\noindent This led Varadarajan and Varadhan to generalize the existence
of the minimal type to $\lie{g} = \lie{sl}_n(\C)$ for all $n$, while they
were at the Indian Statistical Institute, Kolkata. (See \cite{Var} for a
very nice historical account of the development of \cite{PRV2}.)
Subsequently in \cite{PRV2}, the authors extended the result to arbitrary
semisimple $\lie{g}$, and obtained the previously sought-for existence
and unique multiplicity of the minimal type. Here is their result.

\begin{theorem}[\cite{PRV2}]\label{Tprv2}
Suppose $\lie{g}$ is semisimple, and $\lambda, \mu \in \Lambda^+$. Given
$\nu \in \Lambda^+$, define $\overline{\nu}$ to be the unique
$W$-translate of $\nu$ that lies in $\Lambda^+$. Then
$m^{\overline{\lambda + w_\circ \mu}}_{\lambda,\mu} = 1$. Moreover,
\[ m^\nu_{\lambda, \mu} > 0 \implies \overline{\lambda + w_\circ \mu} \in
\wt V(\nu). \]
\end{theorem}

\noindent (More generally - say by the result from \cite{KLV} stated in
the next proof below - $\wt V(\overline{\lambda + w_\circ \mu}) \subset
\wt V(\nu)$ for all such $\nu$.) Thus, the sought-for minimal type exists
and possesses the desired properties. We will see later, how this leads
to the construction of interesting polynomial maps and
infinite-dimensional Banach space representations of $G$.

For completeness, we remark that the ``maximal type" also exists in
general:

\begin{prop}\label{Pmax}
If $\lie{g}$ is semisimple and $\lambda, \mu \in \Lambda^+$, then
$m^{\lambda + \mu}_{\lambda, \mu} = 1$. Moreover,
\[ m^\nu_{\lambda,\mu} > 0 \implies \nu \in \wt V(\lambda + \mu). \]

\noindent More generally, $\wt V(\lambda) \otimes V(\mu) = \wt V(\lambda
+ \mu)$.
\end{prop}

\noindent Inductively, $\wt \otimes_{i=1}^n V(\lambda_i) = \wt V(\sum_i
\lambda_i)$ if all $\lambda_i \in \Lambda^+$.

\begin{proof}
We only show that $\wt V(\lambda) \otimes V(\mu) \subset \wt V(\lambda +
\mu)$. Note that $\wt V(\lambda) \otimes V(\mu) = \bigcup_\nu \wt
V(\nu)$, where we run over all $\nu \in \Lambda^+$ such that
$m^\nu_{\lambda, \mu} > 0$.
Now from Kostant's results mentioned above, every such $\nu$ is of the
form $\lambda + \mu'$, where $\mu' \in \wt V(\mu)$. Hence it suffices to
prove that
\[ \mu' \in \wt V(\mu),\ \lambda + \mu' \in \Lambda^+ \implies \wt
V(\lambda + \mu') \subset \wt V(\lambda + \mu). \]

We now quote a result from \cite{KLV}, which says that given $\lambda,
\mu \in \Lambda^+$, $\lambda - \mu \in \Z_{\geq 0} \Pi$ if and only if
$\conv(W \mu) \subset \conv(W \lambda)$, where $\conv$ denotes the convex
hull.
Applying this with $\mu \leadsto \lambda + \mu', \lambda \leadsto \lambda
+ \mu$, $\conv(W(\lambda + \mu')) \subset \conv(W(\lambda + \mu))$.
Now given $\nu' \in \wt V(\lambda + \mu')$, it is clear that $(\lambda +
\mu) - \nu' \in \Z_{\geq 0} \Pi$.
Recall \cite[Theorem 7.41]{Ha}, which says that for all $\lambda \in
\Lambda^+$, $\displaystyle \wt V(\lambda) = (\lambda - \Z \Pi) \cap
\conv(W \lambda)$. Applying this first with $\lambda \leadsto \lambda +
\mu'$ and then with $\lambda \leadsto \lambda + \mu$, we get that $\nu'
\in \conv(W(\lambda + \mu')) \subset \conv(W(\lambda + \mu))$, so $\nu'
\in \wt V(\lambda + \mu)$ as desired.
\end{proof}

\subsection{The (K)PRV conjecture and generalized PRV components}

We now discuss a vast generalization of Theorem \ref{Tprv2}, which was
conjectured by Parthasarathy et al, extended by Kostant and refined by
Verma, and proved by Kumar in \cite{Kum1,Kum3}. The ``PRV Conjecture" has
been the subject of much study and numerous papers in the literature, and
continues to attract interest, as we point out below.

To state the conjecture, recall the following facts from above: given
$\lambda, \mu \in \Lambda^+$,
\begin{itemize}
\item $m^{\lambda + 1 \cdot \mu}_{\lambda,\mu} = 1$.

\item Equation \eqref{E43} says:
$\lambda + w_\circ \mu \in \Lambda^+ \implies m^{\lambda + w_\circ
\mu}_{\lambda,\mu} = 1$.
\end{itemize}

\noindent There is a common generalization of these assertions to
arbitrary $w \in W$, which is mentioned in \cite{Kum5,PY}. Namely, given
$\lambda, \mu \in \Lambda^+$ and $w \in W$,
\[ \lambda + w \mu \in \Lambda^+ \implies m^{\lambda + w
\mu}_{\lambda,\mu} = 1. \]

\noindent It is natural to ask what happens when $\lambda + w \mu \notin
\Lambda^+$. In light of Theorem \ref{Tprv2}, a natural guess would be to
ask if $m^{\overline{\lambda + w \mu}}_{\lambda,\mu} = 1$, or at least,
if this multiplicity is positive. This is known as the {\it PRV
Conjecture} in the literature.

Kostant significantly strengthened the PRV conjecture in the following
way. Recall that the formal character of each finite-dimensional module
$V(\lambda)$ is $W$-invariant, which implies that for all $w \in W$,
$\dim V(\lambda)_{w \lambda} = 1$. Suppose $v_{w \lambda}$ and $v'_{w
\mu}$ are nonzero vectors that span the ``extremal weight
spaces" $V(\lambda)_{w \lambda}$ and $V(\mu)_{w \mu}$ respectively, for
all $\lambda, \mu \in \Lambda^+$ and $w \in W$. It is then clear that
$v_\lambda \otimes v'_\mu$ generates the copy of the ``maximal type"
$V(\lambda + \mu)$ inside $V(\lambda) \otimes V(\mu)$. Now consider the
minimal type: is it generated by $v_\lambda \otimes v'_{w_\circ \mu}$?
The answer is no: in fact, this vector generates the entire module!
In other words, $\lie{Ug}(v_\lambda \otimes v'_{w_\circ \mu}) =
V(\lambda) \otimes V(\mu)$. Moreover, Theorem \ref{Tprv2} says that
exactly one copy of $V(\overline{\lambda + w_\circ \mu})$ sits in it.

It is now possible to generalize both of these statements. Given any $w
\in W$, consider the $\lie{g}$-submodule generated by $v_\lambda \otimes
v'_{w \mu}$. Does it contain a (unique) copy of $V(\overline{\lambda + w
\mu})$? This is the subject of the {\it KPRV conjecture}, which was
formulated by Kostant and proved by Kumar in \cite{Kum1} in the
semisimple case.

\begin{theorem}[\cite{Kum1,Kum2,Mat1}]\label{Tkprv}
Suppose $\lie{g}$ is semisimple, $\lambda, \mu \in \Lambda^+$, and $w \in
W$. Then the module $V(\overline{\lambda + w \mu})$ appears with
multiplicity $1$ in the submodule $\lie{Ug}(v_\lambda \otimes v'_{w
\mu})$ of $V(\lambda) \otimes V(\mu)$.
\end{theorem}

\noindent Note that a part of Theorem \ref{Tprv2} is just the special
case $w = w_\circ$ of this result. Moreover, the components
$V(\overline{\lambda + w \mu})$ are known as {\it generalized PRV
components}.

The KPRV conjecture was also extended to symmetrizable Kac-Moody Lie
algebras by Mathieu. Given a symmetrizable generalized Cartan matrix $A$,
one can again define the associated Kac-Moody Lie algebra $\lie{g}(A)$
over a field $k$. When $\chr k = 0$, one defines the above notions of
dominant integral weights $\Lambda^+$ and $\overline{\lambda}$, as well
as simple highest weight $\lie{g}(A)$-modules $L(\lambda)$ corresponding
to any weight $\lambda$.
In \cite{Mat1}, Mathieu defined an associated Kac-Moody group over an
arbitrary field $k$ (after earlier work by Kac, Moody, Peterson, and
Tits) using the formalism of ind-schemes. His work led him to prove
Theorem \ref{Tkprv} for $\lie{g}(A)$. (Kumar also proved this case under
the assumption that $\lambda$ is regular, in \cite{Kum2}.)

Other proofs of the (K)PRV conjecture have since appeared in the
literature (this is from \cite{Kum5}). For example, Polo had proved the
PRV conjecture in type $A$ in \cite{Po}. Rajeswari \cite{Ra} gave a proof
for classical $\lie{g}$ using Standard Monomial Theory; Littelmann did so
using his LS-path model (which generalizes the Littlewood-Richardson rule
using tableaux for $\lie{gl}(n)$, to symmetrizable Kac-Moody algebras -
see \cite{Li}); and Lusztig's work on the intersection homology of
generalized Schubert varieties associated to affine Kac-Moody groups also
provides a proof.

Here is another related result. Note that if $w' \leq w$ in the Bruhat
order, then $\lie{Ug}(v_\lambda \otimes v'_{w' \mu}) \subset
\lie{Ug}(v_\lambda \otimes v'_{w \mu})$. This follows inductively from
the case when $w = s_i w' > w'$, which is proved inside the module
$V_i((w'\mu)(h_i))$ over the ``$\alpha_i$-copy" of $\lie{sl}_2(\C)$, by
showing that $v_\lambda \otimes v'_{w' \mu} = e_i^{(w' \mu)(h_i)}
(v_\lambda \otimes v'_{w \mu})$. One can now ask if the component
$V(\overline{\lambda + w \mu})$ occurs in $\lie{Ug}(v_\lambda \otimes
v'_{w' \mu})$ for some $w' < w$, or if $\lie{Ug}(v_\lambda \otimes v'_{w
\mu})$ is the ``first time" that it occurs in $V(\lambda) \otimes
V(\mu)$.

\begin{prop}[\cite{Kum1}]
Given regular $\lambda, \mu \in \Lambda^+$ (i.e., $\lambda(h_i),
\mu(h_i)> 0 \ \forall i \in I$) and $w' < w$ in the Bruhat order on $W$,
the $\lie{g}$-module $V(\overline{\lambda + w \mu})$ does not occur in
$\lie{Ug}(v_\lambda \otimes v'_{w' \mu})$.
\end{prop}

We end this part by mentioning two further directions in which the
(original) PRV conjecture has been generalized very recently. Suppose $G
\subset \widehat{G}$ are complex connected reductive groups, such that
$\widehat{W}$ is the Weyl group of $\widehat{G}$ and various subgroups
are ``compatible" with the inclusion $: G \hookrightarrow \widehat{G}$
(e.g., $B \subset \widehat{B}$, $T \subset \widehat{T}$, $W \subset
\widehat{W}$).
Given a dominant integral weight $\widehat{\lambda}$ for $\widehat{G}$
and $\widehat{w} \in \widehat{W}$, does the simple highest-weight
(finite-dimensional) module $V_G(\overline{\rho(\widehat{w}
\widehat{\lambda})})$ with extremal weight $\rho(\widehat{w}
\widehat{\lambda})$ occur inside $V_{\widehat{G}}(\widehat{\lambda})$
(upon restricting this to $G$)? Here, $\rho$ is the restriction of a
weight from $\widehat{G}$ to $G$. For example, the classical PRV
conjecture uses $\widehat{G} = G \times G$ containing the diagonal copy
of $G$, and
\[ \qquad \widehat{W} = W \times W, \qquad \widehat{\lambda} =
(\lambda,\mu), \qquad \widehat{w} = (1, w), \qquad \rho(\lambda,\mu) =
\lambda + \mu. \]

\noindent The above question is addressed in great detail for more
general pairs $G \subset \widehat{G}$ in the recent papers
\cite{MPR1,MPR2}, under the assumption that $\widehat{G} / G$ is
``spherical of minimal rank".

Finally, Hayashi has proved a quantum counterpart of the PRV conjecture
in \cite{Hay} in the context of fusion rules for
$\widehat{\lie{sl}_3}(\C)$ and the moduli space of $SU(3)$-flat
connections on a pair of pants. These references are intended to
reinforce upon the reader that the PRV conjecture is an extremely
well-studied result, with connections to several other settings in
representation theory and beyond.

\subsection{Tensor product multiplicities, revisited}

We now return to the original question in this section, of computing
Littlewood-Richardson coefficients. As shown above, several results and
formulae have been proposed over the years. Additionally, various other
approaches have appeared more recently in the literature. To name but a
few: Littelmann's LS-path model, Lusztig's approach using canonical
bases, and Kashiwara's use of crystals. See \cite{BZ,Ka,Kum5,Li,Lu} for
references and results.

Recall that the basic questions involving tensor product multiplicities
are: (a) when are the $m^\nu_{\lambda, \mu}$ positive, and (b) computing
the $m^\nu_{\lambda,\mu}$. Results by Kostant, or the PRV conjecture,
address the first question, while the PRV theorem and results by Brauer
and Steinberg discuss the second one. As the above references and related
results show, the work \cite{PRV2} has had quite an influential
contribution in this regard.

We conclude this section with a few additional results in this direction
(see \cite{Kum5}) that exhibit new components, or in some cases, even
obtain a complete decomposition of the tensor product. The first is a
refinement of the above PRV conjecture. One can ask if
$m^{\overline{\lambda + w \mu}}_{\lambda, \mu} = 1$ for all $w$, since it
is so for $w = 1, w_\circ$ from above. This claim turns out to be false -
in fact, Verma produced counterexamples for every $\lie{g}$ of rank 2
(i.e., $|\Pi| = 2$), by choosing $\lambda = \mu = \rho = \sum_{i \in I}
\varpi_i$.

This led Verma to refine the PRV conjecture as follows. The refined
statement was also proved by Kumar.

\begin{theorem}[\cite{Kum3}]
Given $\lambda, \mu \in \Lambda^+$ and $w \in W$, define $W_\lambda$ to
be the stabilizer subgroup of $\lambda$ in $W$, and the map $\eta :
W_\lambda \backslash W / W_\mu \to \Lambda^+$ via: $\eta(W_\lambda w
W_\mu) = \overline{\lambda + w \mu}$.
Then $\displaystyle m^{\overline{\lambda + w \mu}}_{\lambda,\mu} \geq \#
\eta^{-1}( \eta(W_\lambda w W_\mu))$.
\end{theorem}

\noindent Of course, if $\lambda, \mu$ are both regular (i.e., $(\lambda,
\alpha)$ and $(\mu, \alpha)$ are both nonzero for all roots $\alpha$),
then $W_\lambda = W_\mu = \{ 1 \}$.

Second, a related result from \cite{Kum3} is able to determine {\it all}
the multiplicities when $\wt V(\mu) = W \mu$ is a single orbit (i.e.,
$\mu$ is minuscule). In this case,
\[ V(\lambda) \otimes V(\mu) \cong \bigoplus_{\overline{w} \in W / W_\mu
: \lambda + w \mu \in \Lambda^+} V(\lambda + w \mu), \]

\noindent where each factor occurs with multiplicity $1$. There are
exactly $\# W_\lambda \backslash W / W_\mu$ components.

Kumar also shows the following result in \cite{Kum4}: suppose $\beta \in
R^+$ is such that $\lambda + \mu - \beta \in \Lambda^+$, and such that
$\beta - \alpha_i \notin R^+ \cup \{ 0 \}$ whenever $\lambda(h_i)$ or
$\mu(h_i) = 0$. Then $m^{\lambda + \mu - \beta}_{\lambda, \mu} > 0$.
A similar result can be found in the recent work \cite{MPR1}, where the
authors demonstrate new components of the form $w_1 \lambda + w_2 \mu - k
\alpha_i$ for some $w_1, w_2 \in W$ and $i \in I$.

Finally, Dimitrov and Roth have worked with restrictions of line bundles
from the square $G/B \times G/B$ of the flag variety to the diagonally
embedded copy. (Here, $G$ is a connected reductive algebraic group with
$\Lie(G) = \lie{g}$, and, $B \subset G$ is a Borel subgroup with $\Lie(B)
= \lie{h} \oplus \lie{n}^+$; Kumar's proofs in \cite{Kum1} involved a
study of similar objects.) They study special components $V(\nu)$ of the
tensor product modules $V(\lambda) \otimes V(\mu)$, that arise out of
cohomological reasons. The authors mention in \cite{DR} that these {\it
cohomological components} automatically turn out to be generalized PRV
components satisfying: $m^{k \nu}_{k \lambda, k \mu} = 1$ for all $k \in
\N$. They go on to prove the converse implication when $G$ is a classical
group, as well as in other cases.

\section{Irreducible Banach space representations}

We now continue the discussion prior to the preceding section, about
constructing irreducible $\overline{\lie{g}}$-admissible Banach
(actually, Hilbert) space representations of $G$. In \cite{PRV2}, having
proved that finite-dimensional $\lie{g} \times \lie{g}$-modules have
minimal types, the authors proceed to construct other such
representations (possibly infinite-dimensional), in a completely
different manner. Their construction depends heavily on the work of
Harish-Chandra \cite{Har1}--\cite{Har4}. These representations
$\pihat_{\lambda, \nu}$ are defined on subquotients of a Hilbert space.

\subsection{The work of Harish-Chandra}

In order to outline the construction of these $\lie{g}$-modules by
Harish-Chandra and by Parthasarathy et al, additional notation is needed.
Let $K \subset G$ be the maximal compact subgroup of a complex connected
semisimple Lie group, and let $\lie{h}_0 \subset \Lie(K)$ be a Cartan
subalgebra. Now define $M := \exp( \sqrt{-1} \cdot \lie{h}_0) \subset K$
to be the corresponding Cartan subgroup. For each $\nu \in \Lambda$,
define $\sigma_\nu$ to be the unique character of $M$ that sends
$\exp(\sqrt{-1} \cdot h)$ to $\exp(\sqrt{-1} \cdot \nu(h))$ for all $h
\in \lie{h}_0$. Then $\nu \mapsto \sigma_\nu$ is an isomorphism of
$(\Lambda, +)$ onto the character group $\widehat{M}$ of $M$.

Now let $\lie{H} := L^2(K, \C, \mu)$, where $\mu$ denotes the
(normalized) Haar measure on the compact group $K$. This is a
representation of $M$ under the right-regular action: $(m \cdot f)(k) :=
f(k m^{-1})$. Given $\nu \in \Lambda$, define the $\nu$-weight subspace
of $\lie{H}$ as follows:
\[ \lie{H}(\nu) := \{ f \in \lie{H} : m \cdot f = \sigma_{-\nu}(m)f\
\forall m \in M \}. \]

\noindent Then $\lie{H}$ decomposes as the direct sum of the
$\lie{H}(\nu)$ over all $\nu \in \Lambda$. Moreover, given $\xi \in
\lie{h}^*$, Harish-Chandra had previously defined and studied a
$G$-module structure $\pi_\xi$ on $\lie{H}$ in \cite{Har2}--\cite{Har4}.
It turns out that every $\lie{H}(\nu)$ is a submodule of $\lie{H}$ under
this structure; define $\pi_{\xi,\nu}$ to be this representation. Here
are some of the properties of these modules that are used in \cite{PRV2}.

\begin{theorem}[Harish-Chandra]\label{Thar}
Fix $\xi \in \lie{h}^*$ and $\nu \in \Lambda$.
\begin{enumerate}
\item For all $\mu \in \Lambda^+$, $[\pi_{\xi,\nu} : \vg(\mu)] = \dim
\vg(\mu)_\nu$. In particular, $[\pi_{\xi,\nu} : \vg(\overline{\nu})] =
1$.

\item The representation $\pi_{\xi,\nu}$ has an infinitesimal character
(of $Z(\lie{U}(\lie{g} \times \lie{g}))$).

\item $\pi_{\xi,\nu}$ possesses a distributional character
$\Theta_{\xi,\nu}$, which is a locally summable function that is analytic
on the dense open subset of regular points of $G$. Moreover,
$\Theta_{\xi,\nu} = \Theta_{\xi',\nu'}$ if and only if there exists $w
\in W$ such that $\xi' = w \xi, \nu' = w \nu$.

\item If $\xi$ is restricted to lie in the real subspace $\lie{R}$ of all
weights that take purely imaginary values on $\lie{h}_0$, then
$\pi_{\xi,\nu}$ is always unitary, and almost always irreducible, say
whenever $\xi \in \lie{R}_\nu \subset \lie{R}$ (for each $\nu$). In
particular, if $\xi \in \lie{R}_\nu$ and $w \in W$, then $\pi_{\xi,\nu}
\cong \pi_{w \xi, w \nu}$.
\end{enumerate}
\end{theorem}

Although it is not specifically mentioned in \cite{PRV2}, it is actually
possible to compute the central character of $\pi_{\xi,\nu}$ - and this
has a very familiar expression. See Section \ref{Schar}.

\subsection{Constructing the representations $\pihat_{\lambda,\nu}$}

For his subquotient theorem, Harish-Chandra identified two closed
subspaces $\lie{H}''_\xi(\nu) \subset \lie{H}'_\xi(\nu) \subset
\pi_{\xi,\nu}$, such that the quotient of the larger of them by the
smaller one is an irreducible $G$-module. Obtaining a greater
understanding of these subquotients of $\pi_{\xi,\nu}$ was one of the
main motivations behind \cite{PRV2}; when $G$ is a complex semisimple
group, the authors are indeed able to describe these subspaces more
easily than Harish-Chandra in the real case. We start with this
description. The remainder of this entire section is based on
\cite[Section 2.4]{PRV2}.

Recall from Theorem \ref{Thar} that $[\pi_{\xi,\nu} :
\vg(\overline{\nu})] = 1$. Thus, define $\lie{H}'_\xi(\nu)$ to be the
smallest closed $G$-submodule of $\pi_{\xi,\nu}$ containing the unique
copy of $\vg(\overline{\nu})$. Inside this, define $\lie{H}''_\xi(\nu)$
to be the sum of all closed $G$-submodules $M \subset \lie{H}'_\xi(\nu)$
such that $M \cap \vg(\overline{\nu}) = 0$. Then $\lie{H}''_\xi(\nu)$ is
a maximal submodule of $\lie{H}'_\xi(\nu)$, and this leads to the
irreducible $G$-representations $\pihat_{\lambda,\nu} :=
\lie{H}'_\xi(\nu) / \lie{H}''_\xi(\nu)$, where $\lambda :=
\frac{1}{2}(\xi + \nu) - \rho$ runs over all of $\lie{h}^*$ as well.
Clearly, $[\pihat_{\lambda, \nu} : \vg(\mu)] \leq \dim \vg(\mu)_\nu\
\forall \mu \in \Lambda^+$.

The space $\pihat_{\lambda,\nu}$ was shown in \cite{PRV2} to have the
following properties:
\begin{itemize}
\item $\pihat_{\lambda,\nu}$ is an irreducible subquotient of
$\pi_{\xi,\nu} \subset \lie{H} = L^2(K,\C,\mu)$, hence it too is defined
on a Hilbert space. Moreover, $\pi_{\xi,\nu}$ is irreducible if and only
if $\pihat_{\lambda,\nu} \cong \pi_{\xi,\nu}$, if and only if
$[\pihat_{\lambda,\nu} : \vg(\mu)] = \dim \vg(\mu)_\nu$ for all $\mu \in
\Lambda^+$.

\item $\pihat_{\lambda,\nu}$ is an object of $\scrc(\lie{g} \times
\lie{g}, \overline{\lie{g}})$, with minimal type component
$\overline{\nu} \in \Lambda^+ \cap W \nu$. Moreover, $[\pihat_{\lambda,
\nu} : \vg(\overline{\nu})] = 1$.

\item $\pihat_{\lambda,\nu}$ has the same infinitesimal character as
$\pi_{\xi,\nu}$, where $\lambda = \frac{1}{2}(\xi + \nu) - \rho$.
\end{itemize}

\noindent Note that if $\nu' \notin W \nu$, then $\pihat_{\lambda',
\nu'}$ and $\pihat_{\lambda, \nu}$ cannot be isomorphic by Equation
\eqref{Ehar}, because their minimal types are $\overline{\nu'} \neq
\overline{\nu}$ respectively.

From above, the modules $\pihat_{\lambda,\nu}$ admit infinitesimal
characters. It is clear that the highest weight modules $V(\lambda,\mu)$
also admit such characters. Moreover, both of these are families of
simple objects in $\scrc(\lie{g} \times \lie{g}, \overline{\lie{g}})$.
Therefore it is natural to ask if $\pihat_{\lambda,\nu}$ is
finite-dimensional for some values of the parameters - and if all
finite-dimensional simple modules $V(\lambda,\mu)$ are thus covered.

To answer these questions (affirmatively!), Parthasarathy et al studied
the ``key homomorphisms" $\eta_{V,\cald} : \Omega \to \C$ in greater
detail, by relating them to certain homomorphisms ${\bf h}^\Pi : \Omega
\to P(\lie{h}^* \times \lie{h}^*)$. These homomorphisms are the subject
of the next subsection.

\subsection{Constructing the polynomial-valued maps ${\bf h}^{\Pi'}$}

Recall the ``key homomorphism" $\eta_{V,\cald} : \Omega \to \C$, that is
defined whenever a simple $\overline{\lie{g}}$-module $\cald$ arises with
multiplicity one in a simple object $V$ of $\scrc(\lie{g} \times \lie{g},
\overline{\lie{g}})$.
It turns out that there is an explicit construction of the map
$\eta_{\pihat_{\lambda,\nu}, \overline{\nu}}$ via a different
homomorphism ${\bf h}^{\Pi'}(-; \lambda, \nu)$, which we now present. We
explicitly compute both of these maps below in the example of $\lie{g} =
\lie{sl}_2(\C)$, to show that they are equal. This material discusses
\cite[Sections 2.3, 2.4]{PRV2}.

To construct the map $\h^{\Pi'}$, some more notation is needed. Given $X
\in \lie{g}$, define
\[ \one{X} := X \otimes 1, \qquad \two{X} := 1 \otimes X, \qquad
\overline{X} := \one{X} + \two{X}, \]

\noindent and similarly, $\one{\lie{g}}, \two{\lie{g}} \subset \ghat =
\lie{g} \oplus \lie{g}$, as well as $\one{\lie{h}}$ and so on. Then $X
\mapsto \overline{X}$ extends to an isomorphism of associative algebras
$: \lie{Ug} \to \lie{U} \overline{\lie{g}}$, and similar statements hold
for $\two{\lie{g}}, \one{\lie{h}}$, etc.

Now define $\qhat := \one{\lie{n}}^+ \oplus \two{\lie{n}}^-$. Note that
this is the ``positive part" of the triangular decomposition of $\lie{g}
\times \lie{g}$, if we define $\widehat{\Pi} := \one{\Pi} \coprod
-\two{\Pi}$. This choice of simple roots for $\ghat$ comes from the
consideration of the conjugation on $\lie{g}$ with respect to a compact
form; see a previous footnote.
Now $\ghat = \overline{\lie{g}} \oplus \one{\lie{h}} \oplus \qhat$, so by
the PBW Theorem,
\[ \lie{U} \ghat \cong (\lie{U} \overline{\lie{g}} \otimes \lie{U}
\one{\lie{h}}) \oplus (\lie{U} \ghat) \qhat \]

\noindent as $\C$-vector spaces. Note that every $H \in \Sym \lie{h}$ is
a polynomial on $\lie{h}^*$ as follows: write $H = p(\{ h_i : i \in I
\})$ for some polynomial $p$. Then $H(\lambda) = p(\{ \lambda(h_i) : i
\in I \})$. This also applies to $H \in \Sym \one{\lie{h}}$ or $\Sym
\overline{\lie{h}}$, for instance, via the obvious isomorphisms mentioned
above.
Similarly, define $\h^{\n} := \prod_{i \in I} h_i^{n_i}$ for $\n =
(n_i)_{i \in I} \in \Z_{\geq 0}^I$ (which we write as: $\n \geq {\bf
0}$), and also $\overline{\h}^{\n}, \one{\h}^{\n}$, and $\lambda(\h)^{\n}
= \lambda(\h^{\n}) := \h^{\n}(\lambda)$ from above.

We can now define the homomorphisms in question. Suppose $\omega \in
\Omega$, the centralizer of $\overline{\lie{g}}$ in $\lie{U \ghat}$. Then
there exists a unique $\xi_{\n} \in \lie{Ug}$ for all $\n \geq {\bf 0}$,
such that
\[ \omega \equiv \sum_{\n \geq {\bf 0}} \overline{\xi_{\n}} \otimes
\one{\h}^{\n} \mod (\lie{U} \ghat) \qhat. \]

\noindent Since $[\overline{h}, \omega] = 0$ for all $h \in \lie{h}$, one
checks that $\xi_{\n} \in (\lie{Ug})_0$ for all $\n$. Finally, given any
subset $\Pi' \subset R$ of simple roots for some Borel subalgebra
(equivalently, $\Pi' = w \Pi$ for some $w \in W$), the maps $\h^{\Pi'}$
are defined as follows:
\[ \h^{\Pi'}(\omega) := \sum_{\n \geq {\bf 0}} \beta^{\Pi'}(\xi_{\n})
\otimes \h^{\n} \in \Sym( \lie{h} \times \lie{h}), \qquad
\h^{\Pi'}(\omega; \lambda, \nu) := \sum_{\n \geq {\bf 0}}
\nu(\beta^{\Pi'}(\xi_{\n})) \lambda(\h^{\n})\ \forall \lambda, \nu \in
\lie{h}^*. \]

\noindent It turns out that these polynomials are very familiar
expressions, when $\omega$ is restricted to lie in the center
$Z(\lie{U}(\lie{g} \times \lie{g}))$. We see this in Section \ref{Schar}
below.

\subsection{Relationship between $\pihat_{\lambda,\nu}$ and $\h^{\Pi'}(-;
\lambda,\nu)$}

Recall that there are two classes of irreducible admissible
representations that are constructed in \cite{PRV2}: the
finite-dimensional modules $V(\lambda,\mu)$ for $\lambda, \mu \in
\Lambda^+$, and the Hilbert space representations $\pihat_{\lambda,\nu}$
for $\lambda \in \lie{h}^*$ and $\nu \in \Lambda$. In the former case, we
define $\nu := \lambda + w_\circ \mu \in \Lambda$ (as in Theorem
\ref{Tprv2}); then in both families, the representations all contain the
minimal type $\vg(\overline{\nu})$ with multiplicity $1$.

Now how does one show that the first of the above families is actually
contained inside the second? Similarly, how does one check if two given
representations $\pihat_{\lambda, \nu}$ and $\pihat_{\lambda', \nu'}$ are
equivalent or not? The answer in both cases is to use the homomorphisms
$\h^{\Pi'}$, together with Equation \eqref{Ehar}. More precisely, one
relates the maps $\h^{\Pi'}$ to the homomorphisms $\eta_{\pihat_{\lambda,
\nu}, \overline{\nu}}$. (Note that this does not completely answer the
second question.)

Here are some results from the heart of \cite{PRV2}, in which the authors
begin to address these questions. The proofs use Theorem \ref{Thar}.

\begin{theorem}[\cite{PRV2}]\label{Tprv3}
Suppose $\lambda \in \lie{h}^*$ and $w(\nu) \in \Lambda^+$ for some $\nu
\in \Lambda, w \in W$.
\begin{enumerate}
\item Then $\eta_{\pihat_{\lambda, \nu}, \overline{\nu}}(-) \equiv
\h^{w^{-1} \Pi}(-; \lambda, \nu)$ are homomorphisms $: \Omega \to \C$.

\item The maps $\h^{w \Pi}$ are homomorphisms $: \Omega \to P(\lie{h}^*
\times \lie{h}^*)$ for all $w \in W$. They are $W$-equivariant in the
following sense: for all $\omega \in \Omega$, $w,w' \in W$, and
$\lambda,\nu \in \lie{h}^*$,
\[ \h^{w' w\Pi}(\omega; w' * \lambda, w' \nu) = \h^{w\Pi}(\omega;
\lambda, \nu). \]
\end{enumerate}
\end{theorem}

\noindent A consequence of this result is a ``first step" towards the
classification of the representations $\pihat_{\lambda,\nu}$. (This is
discussed at greater length in Section \ref{Sconc}.)

\begin{cor}\label{Cequiv}
Suppose $(\lambda,\nu), (\lambda',\nu') \in \lie{h}^* \times \Lambda$.
Then $\pihat_{\lambda,\nu} \cong \pihat_{w*\lambda, w \nu}$ for all $w
\in W$, while $\pihat_{\lambda, \nu}$ and $\pihat_{\lambda', \nu'}$ are
not equivalent if $\nu' \notin W \nu$.
\end{cor}

\begin{remark}
Another consequence is the following. Note that since every Verma module
has a unique simple quotient, hence there exists a unique maximal (left)
ideal $\lie{M}_\lambda \subset \lie{Ug}$ containing $\lie{n}^+$ and $\ker
\lambda$. If $\lambda \in \Lambda^+$, then by \cite{Har1},
\[ \lie{M}_\lambda = (\lie{Ug}) \lie{n}^+ + (\lie{Ug}) \ker \lambda +
\sum_{i \in I} \lie{Ug} \cdot f_i^{\lambda(h_i) + 1}. \]

\noindent Thus whenever $w(\nu) \in \Lambda^+$ for $\nu \in \Lambda$ and
$w \in W$, there exists a unique maximal ideal in $\lie{U}
\overline{\lie{g}}$ containing $\ker \nu \subset \overline{\lie{h}}$ and
$\{ \overline{e_\alpha} : \alpha \in w^{-1}(R^+) \}$, where $e_\alpha$
spans $\lie{g}_\alpha$. Call this ideal $\overline{\lie{M}}_\nu$. Now
since $\pihat_{\lambda,\nu}$ is a simple $\lie{g} \times \lie{g}$-module,
it is generated by the $\overline{\lie{g}}$-(maximal) weight vector
$v_{\overline{\nu}} \in \vg(\overline{\nu})_{\overline{\lie{\nu}}}
\subset \pihat_{\lambda,\nu}$.
Moreover, $\Omega$ acts on $\vg(\overline{\nu})$ by $\h^{w^{-1} \Pi}(-;
\lambda, \nu)$. By Equation \eqref{Ehar}, this data uniquely determines
$\pihat_{\lambda,\nu}$ up to isomorphism. This implies the following
result from \cite{PRV1}:

{\it There exists a unique maximal ideal $\lie{M}_{\lambda,\nu} \subset
\lie{U}(\lie{g} \times \lie{g})$ containing $\ker \h^{w^{-1}\Pi}(-;
\lambda, \nu) \subset \Omega$ and $\overline{\lie{M}}_\nu$. Moreover,
$\pihat_{\lambda,\nu} \cong \lie{U}(\lie{g} \times \lie{g}) /
\lie{M}_{\lambda,\nu}$.}
\end{remark}

Note that $\h^{\Pi'} : \Omega \to P(\lie{h}^* \times \lie{h}^*)$ is a
homomorphism for each $\Pi' = w \Pi$. Thus, $\h^{\Pi'}(-; \lambda, \nu)$
is a homomorphism $: \Omega \to \C$ for all $\lambda, \nu \in \lie{h}^*$.
The strategy in \cite{PRV2} for showing this is to restrict to the
Zariski dense subset of $(\lambda,\nu)$ arising from finite-dimensional
modules. The authors prove that if $\mu$ and $\Pi'$ are chosen
``suitably", then $\h^{\Pi'}(-; \lambda, \nu) \equiv
\eta_{V(\lambda,\mu), \overline{\nu}}(-)$ is a homomorphism $: \Omega \to
\C$. Hence so is $\h^{\Pi'}$ at all values of $(\lambda,\nu)$. This
analysis leads to the next topic.

\subsection{Relationship between $V(\lambda,\mu)$ and $\h^{\Pi'}(-;
\lambda,\nu)$}

Consider the other family of simple $\scrc(\lie{g} \times \lie{g},
\overline{\lie{g}})$-modules studied in \cite{PRV2}: the
finite-dimensional $V(\lambda,\mu)$. From above, the minimal type of such
a module is $\overline{\lambda + w_\circ \mu}$. Now consider the
``converse" question: given $\lambda \in \lie{h}^*$ and $\nu \in
\Lambda$, is it possible to produce $\mu \in \Lambda^+$ such that
$V(\lambda,\mu)$ has minimal type $\overline{\nu}$? Supporting evidence
for such a claim is given by the following result, in light of Equation
\eqref{Ehar}.

\begin{theorem}[\cite{PRV2}]\label{Tprv4}
Suppose $\lambda \in \Lambda^+$ and $\nu \in \lambda - \Lambda^+$. Choose
$w \in W$ such that $w(\nu) \in \Lambda^+$ and define $\mu :=
-w_\circ(\lambda - \nu) \in \Lambda^+$. Then $V(\lambda,\mu)$ has minimal
type $\vg(\overline{\nu})$; moreover,
\[ \eta_{V(\lambda,\mu), \overline{\nu}}(-) \equiv \h^{w^{-1} \Pi}(-;
\lambda, \nu). \]
\end{theorem}

\noindent Restating this allows us to answer the ``converse" question,
which can also be found in \cite{Du2}.

\begin{cor}\label{Cprv}
For all $\lambda, \mu \in \Lambda^+$, $\displaystyle V(\lambda, \mu)
\cong \pihat_{\lambda,\lambda + w_\circ \mu}$.
\end{cor}

\begin{proof}
Note that $\nu = \lambda + w_\circ \mu \in \lambda - \Lambda^+
\Leftrightarrow \mu = -w_\circ (\lambda - \nu) \in \Lambda^+$. Thus, set
$\nu := \lambda + w_\circ \mu$ and choose $w \in W$ such that $w(\nu) \in
\Lambda^+$. Then using Theorems \ref{Tprv3} and \ref{Tprv4},
\[ \eta_{\pihat_{\lambda,\lambda + w_\circ \mu}, \overline{\nu}}(-)
\equiv \h^{w^{-1} \Pi}(-; \lambda, \lambda + w_\circ \mu) \equiv
\eta_{V(\lambda, \mu), \overline{\nu}}(-) \]

\noindent on $\Omega$. Moreover, $\overline{\nu}$ is the (multiplicity
one) minimal type component of both irreducible modules, by the PRV
Theorem. The result follows by applying Equation \eqref{Ehar}.
\end{proof}

\begin{remark}
A word of caution: note that $V(\lambda) \otimes V(\mu) \cong V(\mu)
\otimes V(\lambda)$ as $\overline{\lie{g}}$-modules for $\lambda, \mu \in
\lie{h}^*$, since $\lie{Ug}$ is cocommutative. Thus, their minimal types
are also equal (when $\lambda, \mu \in \Lambda^+$): $\overline{\lambda +
w_\circ \mu} = \overline{\mu + w_\circ \lambda}$. This implies that the
action of $Z(\lie{U} \overline{\lie{g}})$ is equal on both modules.
However, $V(\lambda,\mu)$ and $V(\mu,\lambda)$ are non-isomorphic simple
$\lie{g} \times \lie{g}$-modules if $\lambda \neq \mu$. Similarly, the
infinitesimal characters are not equal on all of $Z(\lie{U}(\lie{g}
\times \lie{g}))$ - and hence the key homomorphisms
$\eta_{V(\lambda,\mu), \overline{\nu}}, \eta_{V(\mu,\lambda),
\overline{\nu}}$ do not agree on all of $\Omega$ - unless $\lambda =
\mu$.
\end{remark}

\subsection{Example: the case of $\lie{sl}_2(\C)$}\label{Sexample}

We now verify a part of Theorem \ref{Tprv4} in the special case of
$\lie{g} = \lie{sl}_2(\C)$, in which case it is not hard to compute both
the homomorphisms in question - at least, on a particular finitely
generated subalgebra $\Omega' \subset \Omega$.

For convenience, denote the two factors in $\ghat$ as $\lie{g}_k$ for
$k=1,2$ (as opposed to $\one{\lie{g}}$ and $\two{\lie{g}}$), with bases
$\{ e_k, f_k, h_k \}$. Now for all $0 < n_2 \leq n_1 \in \N$, the
``tensor product" module $V_1(n_1) \otimes V_2(n_2)$ is a simple object
in $\scrc(\lie{g} \times \lie{g}, \overline{\lie{g}})$, of highest weight
$(n_1,n_2)$. Restricted to $\overline{\lie{g}}$, it decomposes according
to the Clebsch-Gordan coefficients:
\[ V(n_1, n_2) \cong_{\overline{\lie{g}}} V(n_1) \otimes V(n_2) \cong
\vg(n_1+n_2) \oplus \vg(n_1+n_2-2) \oplus \dots \oplus \vg(n_1-n_2), \]

\noindent where $\vg(n)$ is a finite-dimensional irreducible
$\overline{\lie{g}}$-module of highest weight $n$ (equivalently, of
dimension $n+1$). Now denote the highest weight generators of $V_i(n_i)$
by $v_{n_1}$ and $v'_{n_2}$ respectively, and define the weight basis
$v_{n_1-2i} := (f^i/i!) v_{n_1}$ of $V_1(n_1)$, with $0 \leq i < \dim
V_1(n_1)$. Similarly define $v'_{n_2-2j} \in V_2(n_2)$. One checks using
Equation \eqref{Esl2} that
\[ v''_{n_1-n_2} := \sum_{j=0}^{n_2} (-1)^j j! \cdot n_1 (n_1-1) \cdots
(n_1-j+1) \cdot v_{n_1-2j} \otimes v'_{2j-n_2} \]

\noindent is a weight vector in $V(n_1, n_2)$ which is killed by $e_1 +
e_2 := e_1 \otimes 1 + 1 \otimes e_2$. Hence it generates the minimal
type (i.e., the PRV component) $\vg(n_1-n_2)$, and we have:
\[ \lambda = n_1, \qquad \mu = n_2, \qquad \nu = \lambda + w_\circ \mu =
n_1 - n_2 \geq 0, \qquad \overline{\nu} = \nu \in \Lambda^+. \]

Note that the subalgebra $\Omega$ that commutes with $\overline{\lie{g}}$
in $\lie{U}(\ghat)$ contains the center of $\lie{U}(\ghat)$. This is
freely generated by the two Casimir operators
\[ \Delta_i := 4 f_i e_i + h_i^2 + 2 h_i = 4 e_i f_i + h_i^2 - 2 h_i. \]

\noindent Clearly, $\Delta_1 \otimes 1$ acts on $V_1(n_1) \otimes
V_2(n_2)$ by the scalar $n_1^2 + 2 n_1$; similarly, $1 \otimes \Delta_2$
acts by $n_2^2 + 2 n_2$.
Moreover, $\overline{\Delta} = 4 (f_1+f_2)(e_1+e_2) + (h_1+h_2)^2 +
2(h_1+h_2)$ lies in $\Omega$ as well; by $\lie{sl}_2$-theory, it acts on
$v''_{n_1-n_2}$ via: $\overline{\Delta} \cdot v''_{n_1-n_2} = ((n_1 -
n_2)^2 + 2(n_1 - n_2)) v''_{n_1-n_2}$.
Since $\overline{\Delta}$ commutes with $\overline{\lie{g}}$, it acts on
$\vg(n_1-n_2)$ by this same scalar. This allows us to determine
$\eta_{V(n_1, n_2), n_1 - n_2}$ - at least, on the subalgebra $\Omega' :=
\C[\Delta_1, \Delta_2, \overline{\Delta}] \subset \Omega$:
\begin{equation}\label{Ecalc}
\eta_{V(n_1,n_2), n_1-n_2}(\Delta_i) := n_i^2 + 2 n_i =
\chi(n_i)(\Delta)\ (i=1,2), \qquad \eta_{V(n_1,n_2),
n_1-n_2}(\overline{\Delta}) := \chi(n_1-n_2)(\Delta).
\end{equation}

Now consider $\h^\Pi(-;\lambda,\nu)$. Note from above that
$\overline{\nu} = \nu = \lambda + w_\circ \mu = n_1 - n_2$, so $w=1$.
Since $\qhat$ is spanned by $e_1$ and $f_2$, projecting the various
$\Delta_i$ onto $\lie{U} \overline{\lie{g}} \otimes \lie{U}
\one{\lie{h}}$ modulo $\qhat$ yields
\[ \Delta_1 = 4 f_1 e_1 + h_1^2 + 2 h_1 \equiv h_1^2 + 2 h_1 \mod
(\lie{U} \ghat) \qhat. \]

\noindent Hence $\h^\Pi(\Delta_1; \lambda, \nu) = \lambda(h_1)^2 + 2
\lambda(h_1) = n_1^2 + 2 n_1$, as in Equation \eqref{Ecalc}. Similarly,
\begin{eqnarray*}
\Delta_2 & = & 4 e_2 f_2 + h_2^2 - 2 h_2 \equiv h_2^2 - 2 h_2 \equiv (h_1
+ h_2)^2 - 2(h_1 + h_2) - h_1^2 - 2 h_1 h_2 + 2 h_1\\
& \equiv & \overline{h}^2 - 2 \overline{h} - 2 \overline{h} \otimes h_1 +
h_1^2 + 2 h_1 \mod (\lie{U} \ghat) \qhat.
\end{eqnarray*}

\noindent Computing $\h^\Pi(\Delta_2; \lambda, \nu)$ amounts to
evaluating this polynomial at $(\nu(\overline{h}), \lambda(h_1)) = (n_1 -
n_2, n_1)$. This yields $n_2^2 + 2 n_2$, as desired. Finally,
$\overline{\Delta} = 2 \overline{f} \overline{e} + \overline{h}^2 + 2
\overline{h}$. To evaluate $\h^\Pi(\overline{\Delta}; \lambda,\nu)$, we
must first apply the Harish-Chandra projection $\beta^\Pi$ to this
expression, which kills the first term. Now evaluating at $(\overline{h},
h_1) = (n_1 - n_2, n_1)$, we obtain $\chi(n_1-n_2)(\Delta)$, as in
Equation \eqref{Ecalc}. Thus, $\h^\Pi(\Delta'; n_1, n_1-n_2) \equiv
\eta_{V(n_1, n_2), n_1-n_2}(\Delta')$ for $\Delta' = \Delta_1, \Delta_2,
\overline{\Delta}$; assuming that both of these are homomorphisms implies
equality on all of $\Omega'$. \qed

\subsection{Infinitesimal characters}\label{Schar}

Although it does not seem to be explicitly mentioned in \cite{PRV2}, the
above facts allow us to compute the infinitesimal characters of the
representations $\pi_{\xi,\nu}$ - or equivalently, of
$\pihat_{\lambda,\nu}$. In fact, we prove a stronger result.

\begin{theorem}\label{Tchar}
For all $\lambda, \nu \in \lie{h}^*$, $w \in W$, and $z \in
Z(\lie{U}(\lie{g} \times \lie{g}))$,
\[ \h^{w \Pi}(z; \lambda, \nu) = \chi(\lambda, \nu - \lambda - 2
\rho)(z), \]

\noindent where $\chi(\lambda, \lambda')$ is the central character for
the $\lie{g} \times \lie{g}$-Verma module $M(\lambda, \lambda')$. Now
given $\xi \in \lie{h}^*$ and $\nu \in \Lambda$, define $\lambda :=
\frac{1}{2}(\xi + \nu) - \rho$.
Then $\chi_{\pi_{\xi,\nu}} = \chi_{\pihat_{\lambda,\nu}} = \chi(\lambda,
\nu - \lambda - 2 \rho)$.\footnote{This assertion about central
characters of Harish-Chandra modules can be found in \cite{Du2}, for
instance.}
Moreover, $\chi_{\pi_{w\xi,w\nu}} = \chi_{\pihat_{w * \lambda, w \nu}} =
\chi(w * \lambda, w * (\nu - \lambda - 2 \rho)) = \chi(\lambda, \nu -
\lambda - 2 \rho)$.
\end{theorem}

\noindent For example, consider the situation where
$\pihat_{\lambda,\nu}$ is finite-dimensional. Thus, $\nu \in \lambda -
\Lambda^+$, and $\pihat_{\lambda,\nu} = V(\lambda, \mu)$, where $\mu :=
-w_\circ (\lambda - \nu) \in \Lambda^+$. Thus, $\nu = \lambda + w_\circ
\mu$. In this case, the result says that the second component of the
central character is:
\[ \chi(\nu - \lambda - 2 \rho) = \chi(\lambda + w_\circ \mu - \lambda -
2 \rho) = \chi(w_\circ \mu + w_\circ \rho - \rho) = \chi(w_\circ * \mu),
\]

\noindent since $w_\circ \rho = - \rho$. Hence $\chi_{V(\lambda,\mu)} =
\chi(\lambda,w_\circ * \mu) = \chi(\lambda, \mu)$ (by Theorem \ref{Thc}),
as expected.

\begin{proof}
We use a ``Zariski density" argument as in \cite{PRV2}, that Varadarajan
attributes to Harish-Chandra in \cite{Var}. By Harish-Chandra's Theorem
\ref{Thc}, every central character of $Z(\lie{U}(\lie{g} \times
\lie{g}))$ is of the form $\chi(\mu_1, \mu_2)$ for some $\mu_i \in
\lie{h}^*$. Moreover, $Z(\lie{U}(\lie{g} \times \lie{g})) \subset
\Omega$, so we can evaluate $\h^{w^{-1} \Pi}(-; \lambda, \nu)$ on it. For
all $z \in Z(\lie{Ug})$, one uses the definitions to check that
$\h^{w^{-1} \Pi}(\one{z}; \lambda,\nu) = (\nu \otimes \lambda)(1 \otimes
\beta^\Pi(\one{z})) = \chi(\lambda)(z)$ for all $\lambda, \nu \in
\lie{h}^*$ and all $\Pi'$.

For $\two{z}$ in the second copy of the center, first suppose that $\nu
\in \Lambda^+$. Write the Harish-Chandra projection of $z$, but now using
the decomposition $\lie{Ug} = \lie{Un}^+ \otimes \lie{Uh} \otimes
\lie{Un}^-$. In other words, compute $\beta^{w_\circ \Pi}(z)$, since
$\two{z} \equiv \two{(\beta^{w_\circ \Pi}(z))} \mod (\lie{U} \ghat)
\qhat$. Now use the basis $h_i\two{}$ of $\two{\lie{h}}$ to write out the
above as a polynomial:
\[ \two{(\beta^{w_\circ \Pi}(z))} = p(\{ h_i\two{} : i \in I \}) =
p(\{\overline{h_i} - h_i\one{} : i \in I \}). \]

\noindent Hence $\h^\Pi(\two{z}; \lambda, \nu)$ involves acting by $\nu$
on $\overline{h_i}$ and by $\lambda$ on $h_i\one{}$. Now recall that for
all $z \in Z(\lie{Ug})$, $\overline{z} \in Z(\lie{U}\overline{\lie{g}})
\subset \Omega$. Using these facts, and fixing $\nu' \in \lie{h}^*$,
compute using Theorem \ref{Tprv3}:
\begin{eqnarray*}
\h^\Pi(\two{z}; \lambda,\nu) & = & p(\{ \nu(h_i) - \lambda(h_i) \}) =
p(\{ (\nu - \lambda)(h_i) \}) = (\nu - \lambda)(\beta^{w_\circ
\Pi}(z)) \nu'(1)\\
& = & \h^{w_\circ \Pi}(\overline{z}; \nu', \nu - \lambda) =
\h^\Pi(\overline{z}; w_\circ * \nu', w_\circ(\nu - \lambda)) =
\chi(w_\circ(\nu - \lambda))(z),
\end{eqnarray*}

\noindent since $w_\circ = w_\circ^{-1}$. By Theorem \ref{Thc}, twist
this weight by $w_\circ$; thus for all $(z, \lambda, \nu) \in Z(\lie{Ug})
\times \lie{h}^* \times \Lambda^+$,
\[ \h^\Pi(\two{z}; \lambda, \nu) = \chi(w_\circ * w_\circ(\nu -
\lambda))(z) = \chi(\nu - \lambda - 2 \rho)(z). \]

\noindent Now fix any weight $\lambda$ and any central $z$, and consider
the map $h_{z, \lambda} : \lie{h}^* \to \C$, given by:
\[ h_{z, \lambda}(\nu) := \h^\Pi(\two{z}; \lambda, \nu) - \chi(\nu -
\lambda - 2 \rho)(z). \]

\noindent It is clear that $h_{z,\lambda}$ is a polynomial map, which
vanishes if $\nu$ lies in the Zariski dense subset $\Lambda^+ \subset
\lie{h}^*$. But then $h_{z,\lambda} \equiv 0$ as polynomials. (This is
the aforementioned Zariski density argument; it is analogous to saying
that if a polynomial $p(T_1, \dots, T_n) : \C^n \to \C$ is identically
zero on $z_0 + \Z_{\geq 0}^n$ for some $z_0 \in \C^n$, then $p \equiv
0$.) We conclude that for all $\lambda,\nu \in \lie{h}^*$ and $z \in
Z(\lie{Ug})$, $\h^\Pi(z; \lambda, \nu) = \chi(\lambda, \nu - \lambda - 2
\rho)(z)$. It is also easy to check that the following holds:
\begin{equation}\label{Ehc}
w * (\nu - \lambda - 2 \rho) = w \nu - w*\lambda - 2 \rho\ \forall w \in
W, \nu, \lambda \in \lie{h}^*.
\end{equation}

\noindent Now given any $w \in W$, compute using Theorems \ref{Thc} and
\ref{Tprv3}:
\begin{eqnarray*}
\h^{w^{-1} \Pi}(z; \lambda, \nu) & = & \h^\Pi(z; w * \lambda, w \nu) =
\chi(w * \lambda, w \nu - w * \lambda - 2 \rho)(z)\\
& = & \chi(w * \lambda, w * (\nu - \lambda - 2 \rho))(z) = \chi(\lambda,
\nu - \lambda - 2 \rho)(z).
\end{eqnarray*}

\noindent This proves the main assertion of the theorem; the rest follow
easily. For instance, if $w \nu \in \Lambda^+$, then considering the
action of any central $z$ on the minimal type (and using Theorem
\ref{Tprv3}),
\[ \chi_{\pihat_{\lambda,\nu}}(z) = \eta_{\pihat_{\lambda,\nu},
\overline{\nu}}(z) = \h^{w^{-1} \Pi}(z; \lambda, \nu) = \chi(\lambda, \nu
- \lambda - 2 \rho)(z). \]
\end{proof}

\begin{remark}
A similar result follows from the definitions: given $z \in Z(\lie{Ug})$,
$\lambda, \nu \in \lie{h}^*$, and $w \in W$, $\h^{w^{-1}
\Pi}(\overline{z}; \lambda,\nu) = \chi(w(\nu))(z)$. Note that this was
shown in \cite{PRV2} only when $w(\nu) \in \Lambda^+$. Now since the
action of any $w \in W$ on $\lie{h}^*$ is a linear - hence, polynomial -
map, once again a Zariski density argument can be used to extend the
result to all $\nu \in \lie{h}^*$, for any fixed $z \in Z(\lie{Ug})$.
\end{remark}

In the above result, observe that the above recipe for the central
character is ``$W$-equivariant", in that the last equation in the
statement of the theorem (which is basically Equation \eqref{Ehc}) holds.
One can check that this is not always so: in other words, if we use
$\chi(\lambda, w * (\nu - \lambda - 2 \rho))$ to denote the central
character (via Theorem \ref{Thc}) for arbitrary $w \neq 1 \in W$ - such
as $w = w_\circ$, say. However, one can check that $W$-equivariance does
hold if $w$ is central in $W$.


\subsection{Remarks}

We conclude this section with some remarks. Note that the presentation in
\cite{PRV2} of the material in this section is motivated by the approach
of Harish-Chandra. Thus, it differs somewhat from the presentation in
this article.

More precisely, the philosophy in \cite{PRV2} (as per the historical
account given in \cite{Var}) was to use Harish-Chandra's {\it density
theorem}, which roughly says that among all irreducible $G$-modules
containing a given irreducible finite-dimensional $\lie{k}$-module
$\cald$, the ones that are finite-dimensional form a Zariski dense set.
What this means is that the homomorphisms $\h^{\Pi'}(-; \lambda, \nu)$
correspond to finite-dimensional representations at special lattice
points, as in Corollary \ref{Cprv} - and the set of these lattice points
is Zariski dense in $\lie{h}^* \times \lie{h}^*$. If we replace these by
more general points, then the homomorphisms in question correspond to
infinite-dimensional representations.

In view of this perspective, the approach in \cite{PRV2} was to
``alternate" between the representations and the homomorphisms. Here is
their strategy:
\begin{itemize}
\item Identify and study the minimal type in all finite-dimensional
simple $\ghat$-modules $V(\lambda,\mu)$.

\item Explicitly compute the polynomials $\h^{\Pi'}$ for these modules,
and prove that $\eta_{V(\lambda,\mu),\overline{\nu}}(-) \equiv \h^{w^{-1}
\Pi}(-; \lambda, \nu)$, where $\nu = \lambda + w_\circ \mu$ and
$\overline{\nu} = w(\nu) \in \Lambda^+$.

\item Motivated by this, claim that there are simple modules
$\pihat_{\lambda,\nu}$ in $\scrc(\lie{g} \times \lie{g},
\overline{\lie{g}})$ for all $\lambda \in \lie{h}^*$ and $\nu \in
\Lambda$, with minimal type $\overline{\nu}$; now construct these.

\item Finally, prove that $\eta_{\pihat_{\lambda,\nu},
\overline{\nu}}(-)$ indeed equals $\h^{w^{-1} \Pi}(-; \lambda, \nu)$ for
these modules.
\end{itemize}

\section{The rings $\scrr_{\nu, \Pi'}$ and the PRV determinants}

The next object of study in \cite{PRV2} is the image $\scrr_{\nu, w^{-1}
\Pi}$ of the map $\h^{w^{-1} \Pi}(-; -, \nu) : \Omega \to \Sym \lie{h}$,
where $w(\nu) \in \Lambda^+$. This was done in detail in \cite[Section
3]{PRV2} using deep results of Kostant \cite{Kos2}, such as his
separation of variables theorem \eqref{Ekostant}. These results lead to
the definition and study of the so-called PRV determinants, which we will
discuss below.

\subsection{The rings $\scrr_{\nu, \Pi'}$}

To state the next result, we need some notation. Recall the stabilizer
subgroup $W_\nu$ of a weight $\nu$, as well as the twisted action of the
Weyl group $*$ on $\lie{h}^*$, which transfers to $\lie{h}$ and then
extends to $\Sym \lie{h} = P(\lie{h}^*)$.

\begin{theorem}[\cite{PRV2}]\label{Trnu}
For all $\nu \in \Lambda$ and $w \in W$ such that $w(\nu) \in \Lambda^+$,
$\scrr_{\nu, w^{-1} \Pi} \subset \I_\nu$, where
\[ \I_\nu := \{ p \in P(\lie{h}^*) : w * p = p\ \forall w \in W_\nu \}.
\]

\noindent If $\nu = 0$ or $\nu(h_i) > 0$ for all $i \in I$, then
$\scrr_{\nu, \Pi} = \I_\nu$. For general $\nu \in \Lambda^+$, define
\[ (\lie{Ug})_0(\nu) := \{ a \in (\lie{Ug})_0 : (\ad e_i)^{\nu(h_i) +
1}(a) = 0\ \forall i \in I \}. \]

\noindent Then $\scrr_{\nu, \Pi} = w_\circ * \beta^{\Pi}((\lie{Ug})_0(-
w_\circ \nu))$, and $\scrr_{w \nu, w \Pi} = w * \scrr_{\nu, \Pi}$ for all
$w \in W$.
\end{theorem}

\noindent Note that if $\nu(h_i) > 0$ for all $i$ (i.e., $\nu$ is
regular), then $\I_\nu = P(\lie{h}^*)$. Moreover, $\I_0 =
P(\lie{h}^*)^{(W,*)}$.

\begin{remark}
The authors claim in \cite{PRV2} that they have proved that $\scrr_{\nu,
\Pi} = \I_\nu$ for all $\nu \in \Lambda^+$, using a case-by-case
analysis.
\end{remark}

We also remark that Theorem \ref{Trnu} has the following consequence for
$\nu = 0$ or regular $\nu$.

\begin{cor}[\cite{PRV2}]\label{Czero}
Suppose $\nu(h_i) > 0$ for all $i$, and $\lambda, \lambda' \in
\lie{h}^*$. Then,
\[ \pihat_{\lambda,0} \cong \pihat_{\lambda',0} \Longleftrightarrow
\lambda' \in W * \lambda, \]

\noindent whereas the representations $\pihat_{\lambda,\nu}$ are
inequivalent for all $\lambda$.
\end{cor}

\noindent The proofs of these results are carefully developed in
\cite[Section 3]{PRV2}, via many intermediate lemmas. These lemmas
heavily use results developed by Kostant in \cite{Kos2}, which deal with
the symmetrization map and with finite-dimensional $\lie{g}$-submodules
of $\Sym \lie{g}$ and $\lie{Ug}$. In the next part, we discuss some of
these results, and show how they can be used to define and study the
``PRV determinants".

We end this part by discussing the case of $\lie{sl}_2$, where we
classify all of the modules $\pihat_{\lambda,\nu}$.

\begin{ex}
Suppose $\lie{g} = \lie{sl}_2(\C)$. Using the results of Section
\ref{Sexample}, compute with $\nu \in \Z_{\geq 0}$:
\[ \h^\Pi(\Delta_1; -, \nu) = h_1^2 + 2 h_1, \quad \h^\Pi(\Delta_2; -,
\nu) = \nu^2 - 2 \nu - 2 \nu h_1 + h_1^2 + 2 h_1, \quad
\h^\Pi(\overline{\Delta}; -, \nu) = \nu^2 + 2 \nu. \]

\noindent It is now clear that if $\nu = 0$, then these generate
$\C[h_1^2 + 2 h_1] = \C[h_1]^{(W,*)}$. The above result says that
$\h^\Pi(\omega; -, \nu)$ is an element of this ring for all $\omega \in
\Omega$. Similarly, if $\nu > 0$ (abusing notation), then the above
polynomials already generate all of $\C[h_1]$ as desired.

Also note that we can classify (all equivalences between) the
representations $\pihat_{\lambda,\nu}$: given $(\lambda, \nu) \neq
(\lambda', \nu')$ in $\lie{h}^* \times \Lambda$, we claim:
\begin{equation}\label{Esimples}
\pihat_{\lambda,\nu} \cong \pihat_{\lambda',\nu'} \Longleftrightarrow
(\lambda',\nu') = (-\lambda-2, -\nu).
\end{equation}

\noindent The backward implication follows from Corollary \ref{Cequiv},
and conversely, $\nu' = \pm \nu$. Then the calculations above imply that
$\h^\Pi(\Delta_1; \lambda, \nu) = \h^\Pi(\Delta_1; \lambda', \nu')$,
whence $\lambda' = \lambda, - \lambda - 2$. This shows the claim when
$\nu' = \pm \nu = 0$. Otherwise we may assume that $\nu' = \nu > 0$
(using the backward implication). Now evaluate $\h^\Pi(\overline{\Delta}
+ \Delta_1 - \Delta_2; -, \nu)$ at $\lambda, \lambda'$. Then $\nu \neq 0
\implies \lambda = \lambda'$. \qed
\end{ex}

\noindent These calculations naturally lead to the question of
classifying the representations $\pihat_{\lambda,\nu}$ for general
semisimple $\lie{g}$ - and more generally, the classification of all
simple objects of $\scrc(\lie{g} \times \lie{g}, \overline{\lie{g}})$.
(Recall Corollary \ref{Cequiv}.) These questions will be discussed in the
concluding section of this article.

\subsection{Kostant's separation of variables}

In order to discuss PRV determinants, we need some preliminaries. Recall
the symmetrization map $\bfl$ from \cite{Har2}, which is the unique
linear isomorphism $: \Sym \lie{g} \twoheadrightarrow \lie{Ug}$
satisfying: $\bfl(1) = 1$ and for all $r>0$ and $X_1, \dots, X_r \in
\lie{g}$,
\[ \bfl(X_1 \dots X_r) = \frac{1}{r!} \sum_{\sigma \in S_r} X_{\sigma(1)}
\dots X_{\sigma(r)}. \]

\noindent Also recall that the adjoint action of $\lie{g}$ on itself can
be uniquely extended to derivations $\Phi : \lie{g} \to \Sym \lie{g}$ and
$\Theta : \lie{g} \to \lie{Ug}$. Thus, both of these algebras are
$\lie{g}$-modules, and $\bfl$ is a $\lie{g}$-module isomorphism: $\bfl
\circ \Phi(X) = \Theta(X) \circ \bfl$ for all $X \in \lie{g}$. Moreover,
the centers are isomorphic via $\bfl$. Namely, $\bfl : (\Sym
\lie{g})^{G_0} \twoheadrightarrow Z(\lie{Ug}) = (\lie{Ug})^{G_0}$, where
$G_0$ is the adjoint group of $\lie{g}$.

Now given $\mu \in \Lambda^+$, define the following copies of the
$\lie{g}$-module $V(\mu)$ in $\Sym \lie{g}$ and $\lie{Ug}$:
\[ \scrl^\mu := \Hom_{\lie{g}}(V(\mu), \Sym \lie{g}), \qquad \scrlug{\mu}
:= \Hom_{\lie{g}}(V(\mu), \lie{Ug}) = \{ \bfl \circ M : M \in \scrl^\mu
\}. \]

\noindent In \cite{Kos2}, Kostant showed that these are both free modules
of rank $d_\mu := \dim V(\mu)_0$, over $(\Sym \lie{g})^{G_0}$ and
$Z(\lie{Ug})$ respectively. This is related to Kostant's ``separation of
variables" theorem; see Equation \eqref{Ekostant}, where $\bbh(\lie{g})$
is precisely the image of the harmonic polynomials in $\Sym \lie{g}$
under $\bfl$. (Recall from above that $[\bbh(\lie{g}) : V(\mu)] = \dim
V(\mu)_0 = d_\mu$.)
Moreover, it is possible to choose all of the basis elements $M_i$ (over
the center) for $\scrl^\mu$ such that every $M_i(V(\mu))$ is a subspace
of $\Sym^{q_i(\mu)} \lie{h}$ for some $q_i(\mu) \in \Z_{\geq 0}$. Such
elements are called {\it homogeneous}. The image of a set of homogeneous
generators under $\bfl$ yields a set of $Z(\lie{Ug})$-module generators
for $\scrl^\mu(\lie{Ug})$.

The zero weight spaces in these modules $\scrl^\mu, \scrlug{\mu}$ are of
great interest in \cite{PRV2}. For instance, it is easy to check that for
all $\nu \in \Lambda^+$,
\[ (\lie{Ug})_0(\nu) = \sum_{\mu \in \Lambda^+} \sum_{L \in \scrlug{\mu}}
L(V^+(\mu; 0, \nu)), \]

\noindent where $V^+(\mu; 0, \nu)$ was defined above Theorem \ref{Tprv1}.
(These spaces were used in Theorem \ref{Trnu}.)

\subsection{PRV determinants}

We finally define the {\it PRV determinants}. (This material is taken
from \cite[Sections 3,4]{PRV2}.) Fix $\mu \in \Lambda^+$, and choose a
set $\{ M_1, \dots, M_{d_\mu} \}$ of homogeneous generators for
$\scrl^\mu$. Also choose a basis $\{ v_1, \dots, v_{d_\mu} \}$ of
$V(\mu)_0$. Now define the {\it PRV matrix} as:
\[ \bfk'_\mu := ((\ \beta^\Pi(\bfl(M_i) v_j)\ ))_{1 \leq i, j \leq d_\mu}
\ \in \lie{gl}_{d_\mu}(\Sym \lie{h}), \qquad \bfk'_\mu(\nu) = (( \
\nu((\bfk'_\mu)_{ij})\ )) \in \lie{gl}_{d_\mu}(\C). \]

One of the important results in \cite{PRV2} that has also influenced much
future research is the computation of the determinants of these and
related matrices. In the following preliminary result, recall by Theorem
\ref{Tprv1} that $\dim V^+(\mu; 0, -w_\circ \nu) = m^\mu_{-w_\circ \nu,
\nu} = [\vg(\nu) \otimes \vg(\nu)^* : \vg(\mu)]$.

\begin{prop}[\cite{PRV2}]
For all $\mu, \nu \in \Lambda^+$, ${\rm rank}\ \bfk'_\mu(\nu) = \dim
V^+(\mu; 0, -w_\circ \nu)$. Moreover,
\[ d_\mu > 0, \lambda \in \lie{h}^* \implies [\pihat_{\lambda,0} :
\vg(\mu)] \leq \min( {\rm rank}\ \bfk'_{- w_\circ \mu}(\lambda), {\rm
rank}\ \bfk'_\mu(w_\circ * \lambda) ). \]
\end{prop}

\noindent The proposition is proved using the auxiliary lemmas developed
in \cite[Section 3]{PRV2}, and holds for all dominant integral $\mu,
\nu$. The authors now propose another matrix which turns out to be
nonsingular at all regular points $h \in \lie{h}$. To define this,
suppose $M_i \in \scrl^\mu$, $v_j \in V(\mu)_0$, and $q_i(\mu) = \deg(M_i
v_j) \geq 0$ as above. Now there exists a unique $h_{ij} \in \Sym
\lie{h}$ such that $M_i v_j \equiv h_{ij} \mod \sum_{\alpha \in R} (\Sym
\lie{h}) \lie{g}_\alpha$, and moreover, $h_{ij}$ is homogeneous of degree
$q_i(\mu)$ for all $1 \leq i,j \leq d_\mu$. We now introduce the
following terminology.

\begin{defn}\hfill
\begin{enumerate}
\item Define the matrix $\bfk_\mu := ((h_{ij}))_{1 \leq i,j \leq d_\mu}$.

\item Fix $\alpha \in R^+$, and suppose $\alpha \leftrightarrow
h'_\alpha$ via the Killing form. Now define $h_\alpha :=
(2/\alpha(h'_\alpha)) h'_\alpha$. Let $e_\alpha, f_\alpha$ span
$\lie{g}_\alpha, \lie{g}_{-\alpha}$ respectively, and let
$m_{j,\mu}(\alpha)$ denote the multiplicity of the eigenvalue $j(j+1)$
for the restriction of $f_\alpha e_\alpha$ to $V(\mu)_0$.
Finally, define $m_\mu(\alpha) = \sum_{j>0} m_{j,\mu}(\alpha)$.
\end{enumerate}
\end{defn}

The following is the main theorem of \cite{PRV2} involving PRV
determinants:

\begin{theorem}[\cite{PRV2}]
Fix $\mu \in \Lambda^+$. Viewed (via the Killing form) as a polynomial
function on $\lie{h}$, $\det \bfk_\mu$ is nonzero at each regular point
$h \in \lie{h}$. In particular, $\det \bfk_\mu, \det \bfk'_\mu$ are
nonzero elements of $P(\lie{h}^*)$ of degree $\sum_i q_i(\mu)$.
Moreover, there exist nonzero constants $c_\mu, c'_\mu$ such that:
\[ \det \bfk_\mu = c_\mu \prod_{\alpha \in R^+} h_\alpha^{m_\mu(\alpha)},
\qquad \det \bfk'_\mu = c'_\mu \prod_{\alpha \in R^+} \prod_{j \geq 1} \{
h_\alpha + \rho(h_\alpha) - 1, j \}^{m_{j,\mu}(\alpha)}, \]

\noindent where for all $a \in \lie{Ug}$ and $j \in \N$, $\{ a, j \} :=
(-1)^j j!\ a(a-1) \dots (a-j+1)$. (In particular, $\sum_{i=1}^{d_\mu}
q_i(\mu) = \sum_{\alpha \in R^+} m_\mu(\alpha)$.)
\end{theorem}

\noindent This is a powerful theorem that computes the PRV determinants
in a simple manner. It can be used to compute these determinants
explicitly for simple $\lie{g}$, with $V(\mu)$ the adjoint
representation, for instance. Let us take a simple example: $\lie{g} =
\lie{sl}_2$. First note that $d_\mu > 0$ if and only if $\mu \in 2
\Z_{\geq 0} \varpi$. Suppose this holds, and apply Equation \eqref{Esl2}.
Then, $f e \cdot v_0 = (\mu/2)(\mu/2 + 1)v_0$, so $m_{j,\mu}(\alpha) =
\delta_{j, \mu/2}$ and $m_\mu(\alpha) = 1$ for all even $\mu$. Hence
\[ \det \bfk_\mu = c_\mu h, \qquad \det \bfk'_\mu = c'_\mu \{ h, \mu/2 \}
\in \C^\times \cdot h(h-1) \dots (h - \mu/2 + 1). \]

\noindent It turns out that this is exactly the Shapovalov determinant
for $\lie{sl}_2(\C)$ (up to a scalar). We now show how these are related
to PRV determinants, before moving on to the next section.

\subsection{Annihilators of Verma modules}

Since they were defined and computed in \cite{PRV2}, PRV determinants
have played a role in the study of annihilators of Verma modules and
their simple quotients, in the following manner. In \cite{Du1}, Duflo
proved the following remarkable result: {\em The annihilator of every
Verma module is centrally generated.} In other words, for any complex
semisimple $\lie{g}$,
\begin{equation}\label{Eduflo}
\Ann_{\lie{Ug}} M(\lambda) = \lie{Ug} \cdot \Ann_{Z(\lie{Ug})} M(\lambda)
= \lie{Ug} \cdot \ker \chi(\lambda), \qquad \forall \lambda \in
\lie{h}^*.
\end{equation}

\noindent The proof of this statement requires a nontrivial
algebro-geometric argument from \cite{Kos2}. Thus, it cannot be extended
to the setting of quantum groups $U_q(\lie{g})$, and a new proof was
sought. This was provided by Joseph for quantum groups, but it holds in
the classical setting as well. In this part, we discuss how PRV
determinants play a role in proving Duflo's result.

Given a weight $\mu \in \lie{h}^*$, every weight space $M(\mu)_{\mu -
\beta}$ of a Verma module has a bilinear form defined on it (where $\beta
\in \Z_{\geq 0} \Pi$). To see this, first fix a Lie algebra
anti-involution $\iota : \lie{g} \to \lie{g}$ that fixes $\lie{h}$ and
sends $e_i$ to $f_i$ for all $i \in I$. Then $\iota$ sends
$\lie{g}_\alpha$ to $\lie{g}_{-\alpha}$ for all roots $\alpha$, and also
extends to an algebra anti-involution of $\lie{Ug}$. Now the {\it
Shapovalov form} is defined as follows:
\[ \Sh(b_1, b_2) := \beta^\Pi(\iota(b_1) \cdot b_2) \in \Sym \lie{h},
\qquad \Sh_\mu(b_1 m_\mu, b_2 m_\mu) := \mu(\Sh(b_1, b_2)), \qquad
\forall b_1, b_2 \in \lie{Un}^-, \]

\noindent where $m_\mu \in M(\mu)_\mu$ generates the Verma module. One
shows that the nondegeneracy of this form can be checked on each
individual weight space. Thus for all $\nu \in \Z_{\geq 0} \Pi$, let
$\det \Sh^\nu$ be the determinant of $\Sh$, when restricted to a (fixed)
weight space basis of $(\lie{Un}^-)_{-\nu}$. The Shapovalov form can now
be shown to possess the following properties (see \cite{MP,Sha} for
instance):

\begin{theorem}
For all $\mu \in \lie{h}^*$, $\Sh_\mu$ is a symmetric bilinear form on
$M(\mu)$. $\Sh_\mu(M(\mu)_\nu, M(\mu)_{\nu'})$ is nonzero only when $\nu
= \nu' \in \mu - \Z_{\geq 0} \Pi$. Moreover, the radical of $\Sh_\mu$ is
the unique maximal submodule of $M(\mu)$, so it is nondegenerate if and
only if $M(\mu) \cong V(\mu)$. Finally, there exists a nonzero constant
$c''_\nu$ such that
\[ \det \Sh^\nu = c''_\nu \prod_{\alpha \in R^+} \prod_{j \geq 1}
(h_\alpha + \rho(h_\alpha) - j)^{\calp(\nu - j \alpha)}. \]
\end{theorem}

We note that both the PRV and Shapovalov determinants $\det \bfk'_\mu,
\det \Sh^\nu$ are products of linear factors. However, more holds! (The
rest of this part is from \cite{FL}.)

\begin{theorem}
For all semisimple $\lie{g}$, the set of all linear factors in $\{ \det
\bfk'_\mu : \mu \in \Lambda^+, d_\mu > 0 \}$ and $\{ \det \Sh^\nu : \nu
\in \Z_{\geq 0} \Pi \}$ coincide. Moreover, given $\lambda \in
\lie{h}^*$, the annihilator $\Ann_{\bbh(\lie{g})} V(\lambda)$ is trivial
if and only $\lambda(\det \bfk'_\mu) = 0$ whenever $d_\mu > 0$.
\end{theorem}

Thus, another approach to proving Duflo's ``Verma module annihilator
theorem" \eqref{Eduflo} is to proceed as follows. This is a program
developed by Joseph and his coauthors.
\begin{enumerate}
\item $U := \lie{Ug}$ has a large ``locally finite subalgebra" $F(U) :=
\{ a \in U : \dim (\ad U)a < \infty \}$.

\item A ``Peter-Weyl" type result holds: $F(U) := \bigoplus_{\lambda \in
\Lambda^+} \End_\C V(\lambda)$. (Note that for $U = \lie{Ug}$, this is
proved for $F(U) = U$ using the perfect pairing between left-invariant
differential operators in $\lie{Ug}$ and regular functions on the simply
connected Lie group $G$ for which $\lie{g} = \Lie(G)$, together with the
usual Peter-Weyl Theorem for regular functions on $G$.)
Under this identification, the lifts of the identity elements in the
various summands are a basis $\{ z_\lambda : \lambda \in \Lambda^+ \}$ of
the center; moreover, $Z(U)$ is a polynomial algebra in the generators
$\{ z_{\varpi_i} : i \in I \}$.

\item There exists an $\ad U$-stable submodule $\bbh$ of $F(U)$ such that
the multiplication map $: \bbh \otimes Z(U) \to F(U)$ is an isomorphism
of $\ad U$-modules.\footnote{If $U = \lie{Ug}$, then $F(U) = U$, and this
result is precisely Equation \eqref{Ekostant}. One can also use this and
Example \ref{Eprv} to prove the Peter-Weyl type result mentioned above,
by counting (countably infinite) multiplicities.} Here, $\ad$ is the
standard adjoint action of the Hopf algebra $U$ on itself.

\item Now define the PRV determinants using the above facts, and the
Shapovalov determinants using the anti-involution $\iota$ and the
Harish-Chandra projection $\beta^\Pi$. Then calculate both sets of
determinants and verify their properties as in the above results.

\item Now use the PRV and Shapovalov determinants for any simple
submodule of a Verma module $M(\lambda)$, which is itself a Verma module
with the same central character $\chi(\lambda)$. Some more work now shows
Duflo's result.
\end{enumerate}

The important point is that this approach works not only for $U =
\lie{Ug}$, but also for $U = U_q(\lie{g})$. (Unlike $\lie{Ug}$, $F(U)
\neq U$ in the quantum case.) Thus, Joseph and Letzter proved the quantum
``separation of variables" theorem, and defined and computed the related
PRV determinants. See \cite{JL1,JL2}; also see \cite{FL} for a historical
exposition of this program.

Joseph and others have since extended this approach to affine Lie
algebras. Similarly, Gorelik and Lanzmann \cite{GL} have also carried out
this program for reductive super Lie algebras. They found that the PRV
determinants contained some ``extra factors" compared to the Shapovalov
determinants, and their zeroes are precisely the weights for which the
corresponding Verma module annihilators are not centrally generated.
Thus, the PRV and Shapovalov determinants (or more precisely, their
common zeroes) turn out to yield, in various settings, both an approach
to proving Duflo's Theorem \eqref{Eduflo}, as well as the set of Verma
modules for which it holds.

\subsection{KPRV determinants}

We end this section with a remark. Kostant described certain analogues of
the PRV determinants in \cite{Kos3} involving parabolic subalgebras of
$\lie{g}$; these analogues had applications related to the irreducibility
of principal series representations. Joseph termed these the {\it KPRV
determinants}, and together with Letzter and with Todoric, has defined
such notions for (quantum) semisimple and affine Lie algebras. (See
\cite{Jo3,JL3,JLT,JT} for more on this, including applications to
annihilators of Verma modules.) Thus, the PRV determinants and their
generalizations continue to be a useful and popular subject of research
in several different settings in representation theory.

\section{Representations of class zero}

In \cite[Section 4]{PRV2}, the authors apply the theory previously
developed to carry out a deeper study of a special sub-family of
irreducible admissible $(\lie{g} \times \lie{g},
\overline{\lie{g}})$-modules: the ones of ``class zero". Here is a brief
discussion of these modules and related results.

\begin{defn}
An irreducible $\overline{\lie{g}}$-admissible $G$-representation $V$ is
said to be of {\it class zero} if $[V : \vg(0)] > 0$.
\end{defn}

\noindent (Note that these are usually referred to as ``class one"
representations in the literature - i.e., irreducible admissible
$(G,K)$-modules which have a $K$-invariant vector. In other words, ``class
one" refers to a $K$-eigenvector with simultaneous eigenvalue $1$, while
``class zero" refers to a $\overline{\lie{g}}$-eigenvector with
simultaneous eigenvalue $0$.) The first result says that every
irreducible Harish-Chandra module of class zero is determined by its
central character; in fact, it is of the form $\pihat_{\lambda,0}$ for
some $\lambda \in \lie{h}^*$. Thus, we obtain deeper insights into the
classification of such modules.(See Equation \eqref{Ehar} and Theorem
\ref{Tchar}.)

\begin{theorem}[\cite{PRV2}]\label{Tprv5}
The set of infinitesimal equivalence classes of class zero irreducible
representations $V$ is in bijection with the twisted Weyl group orbits in
$\lie{h}^*$. More precisely, every such $V$ is uniquely determined by its
infinitesimal character $\chi_V$ restricted to $\one{Z(\lie{Ug})}$.
Moreover, given $\chi_V : Z(\lie{U}(\lie{g} \times \lie{g})) \to \C$,
there exists $\lambda \in \lie{h}^*$ such that
\[ \chi_V = \chi(\lambda, -\lambda - 2 \rho), \qquad V \cong
\pihat_{\lambda,0} \cong \pihat_{w * \lambda, 0}\ \forall w \in W. \]

\noindent In particular $[V : \vg(0)] = 1$ and $0$ is the minimal type of
$V$.
\end{theorem}

\begin{ex}
If $\lambda \in \Lambda^+$, then $\chi_{\pihat_{\lambda,0}} =
\chi(\lambda, w_\circ * (-\lambda - 2\rho)) = \chi(\lambda, -w_\circ
\lambda) = \chi_{V(\lambda) \otimes V(\lambda)^*}$. By Theorem
\ref{Tprv2}, the minimal type of $V(\lambda) \otimes V(\lambda)^*$ is
$\lambda + w_\circ(- w_\circ \lambda) = 0$ as well. Therefore
$\pihat_{\lambda,0} \cong V(\lambda) \otimes V(\lambda)^*$ by the above
result. Moreover, every finite-dimensional $\pihat_{\lambda,0}$ is of
this kind, by Corollary \ref{Cprv}.
\end{ex}

We now take a closer look at the multiplicities. When is
$[\pihat_{\lambda,0} : \vg(\mu)] = \dim \vg(\mu)_0$? (That it is at most
$\dim \vg(\mu)_0$ follows from Theorem \ref{Thar}.) More generally, is it
possible to compute the multiplicities for class zero modules? Once
again, the authors were able to achieve this goal in \cite{PRV2}: the
multiplicity equals the rank of a related matrix, which is defined
similar to $\bfk'_\mu$ above.
Namely, given $\mu \in \Lambda^+$ such that $d_\mu > 0$, choose sets of
homogeneous generators $\{ M_1, \dots, M_{d_\mu} \}$ and $\{ M^*_1,
\dots, M^*_{d_\mu} \}$ for the free $(\Sym \lie{g})^{G_0}$-modules
$\scrl^\mu$ and $\scrl^{-w_\circ \mu}$ respectively. Now choose dual
bases $\{ v_k \}$ and $\{ v_k^* \}$ for $V(\mu), V(-w_\circ \mu) \cong
V(\mu)^*$ respectively. The span of
\[ z_{ij} := \sum_k (\bfl(M^*_i) v^*_k)(\bfl(M_j) v_k) \]

\noindent then depends only on $M_i$ and $M_j$. One can now compute the
sought-for multiplicities.

\begin{theorem}[\cite{PRV2}]
For all $\mu \in \Lambda^+$ and $1 \leq i,j \leq d_\mu$, $z_{ij} \in
Z(\lie{Ug})$. Now given a central character $\chi = \chi(\lambda)$ of
$Z(\lie{Ug})$ (where $\lambda \in \lie{h}^*$),
$[\pihat_{\lambda,0} : \vg(\mu)] = {\bf 1}_{d_\mu > 0} \cdot {\rm rank}
((\ \chi(\lambda)(z_{ij})\ ))$.
\end{theorem}

The next result discusses the case when the multiplicities all attain
their upper bounds. This turns out to be an important question from the
point of view of the irreducibility of the induced representations
$\pi_{\xi,\nu}$ discussed earlier; see also \cite{Bru}. The following
result completely answers this question. (See \cite{Du2} for more results
along these lines.)

\begin{theorem}[\cite{PRV2}]
Given $\lambda \in \lie{h}^*$, the following are equivalent:
\begin{enumerate}
\item $\pihat_{\lambda,0}$ is {\em complete}, i.e., $[\pihat_{\lambda,0}
: \vg(\mu)] = \dim \vg(\mu)_0$ whenever $d_\mu > 0$ for $\mu \in
\Lambda^+$.

\item $\pi_{2(\lambda + \rho), 0}$ is irreducible.

\item The matrices $\bfk'_{-w_\circ \mu}(\lambda)$ and $\bfk'_\mu(w_\circ
* \lambda)$ are both invertible whenever $d_\mu > 0$ for $\mu \in
\Lambda^+$.

\item For all roots $\alpha \in R^+$, $\frac{1}{2} \xi(h_\alpha) =
(\lambda + \rho)(h_\alpha) \notin \Z \setminus \{ 0 \}$.
\end{enumerate}
\end{theorem}

This result now holds when $\xi$ attains purely imaginary values on
$\lie{h}_0$ (i.e., $\xi \in \lie{R} \setminus \{ 0 \}$, in the notation
of Theorem \ref{Thar}). This shows the irreducibility of a class of
unitary representations that was studied previously in the complex
semisimple case:

\begin{theorem}[\cite{PRV2}]
The unitary $G$-representations of the principal nondegenerate series (of
Gelfand and Naimark \cite{GN}) that contain a nonzero $K$-invariant
vector, are all irreducible.
\end{theorem}

\section{Conclusion: The classification of irreducible Harish-Chandra
modules}\label{Sconc}


As discussed in previous sections, the paper \cite{PRV2} has led to much
research in several different directions in representation theory. In
this final section, we return to its original motivations. As is evident
from the paper, as well as from much of the contemporary literature, the
fundamental and profound work of Harish-Chandra on semisimple (real) Lie
groups has had an enormous influence on the field of representation
theory. From the objects studied to the methods employed in \cite{PRV2},
the authors have time and again used contributions of Harish-Chandra to
the subject.

We now mention some of the subsequent developments in the program started
by Harish-Chandra, of studying $K$-admissible $G$-representations. For
instance, various results from \cite{Har2,Har3,PRV2} were subsequently
generalized by Lepowsky in \cite{Le}. Moreover, in \cite{Zh1,Zh2},
Zhelobenko classified irreducible admissible representations of complex
semisimple Lie groups, by showing that they always arise as distinguished
quotients of certain principal series representations. This
classification is the subject of this section.

Before moving on to these classification results, we remark that this
program was extended by Langlands in \cite{La}, to the original setting
of real semisimple Lie groups $G_\R$, where Harish-Chandra had introduced
and studied admissible representations. The {\it Langlands
classification} describes how irreducible admissible representations are
quotients of ``generalized principal series", which are induced from
tempered representations on parabolic subgroups of $G_\R$. The work of
Langlands and Harish-Chandra on tempered representations was refined by
Knapp and Zuckerman \cite{KZ}; thus, one now has an explicit
parametrization of the irreducible admissible representations of the
groups $G_\R$. (See also \cite{BB}, which studies more general
irreducible representations than just the admissible ones.)
As this suggests, the legacy of Harish-Chandra is vast and rich, and
lives on in these works and in the subsequent research which it has
inspired.

\subsection{The set of irreducible objects}

We now discuss the classification of all irreducible objects of
$\scrc(\lie{g} \times \lie{g}, \overline{\lie{g}})$. There are two parts
to this discussion: first, to determine a representative set of simple
objects that covers all isomorphism classes; and second, to determine the
equivalences among the objects in this set. In what follows, the final
results will be stated as they appear in Duflo's notes \cite{Du2} on the
subject.

It turns out that the category $\scrc(\lie{g} \times \lie{g},
\overline{\lie{g}})$ is equivalent to a subcategory of the BGG Category
$\calo$. In particular, Harish-Chandra modules have certain properties in
common with Verma modules. For instance, all objects of this category
have finite length, all simple objects have a corresponding central
character, and for a given central character, the simple objects are
indexed by the Weyl group.
More precisely, Beilinson and Bernstein have classified all irreducible
$(\lie{g} \times \lie{g}, \overline{\lie{g}})$-modules with a fixed
infinitesimal character. Here is a special case of their results.

\begin{theorem}[\cite{BB}]\label{Tbbkk}
Given $\lambda, \mu \in \Lambda^+$, the set of isomorphism classes of
irreducible admissible $(\lie{g} \times \lie{g},
\overline{\lie{g}})$-modules with infinitesimal character $\chi =
\chi(\lambda, \mu)$ is in bijection with $W_\lambda \backslash W /
W_\mu$.
\end{theorem}

\noindent For instance, if $\lambda, \mu$ are both regular, then there
are exactly $|W|$ isomorphism classes, while there is a single class if
$\lambda$ or $\mu$ is zero.

Now recall Theorem \ref{Tprv5}, which says that all irreducible
admissible class zero modules are of the form $\pihat_{\lambda,0}$ for
some $\lambda \in \lie{h}^*$. One can similarly ask: is every irreducible
admissible $(\lie{g} \times \lie{g}, \overline{\lie{g}})$-module of the
form $\pihat_{\lambda,\nu}$ for some $(\lambda,\nu) \in \lie{h}^* \times
\Lambda$? The answer turns out to be positive.

\begin{theorem}
Suppose $V$ is an irreducible object of $\scrc(\lie{g} \times \lie{g},
\overline{\lie{g}})$. Then there exist $\lambda \in \lie{h}^*$ and $\nu
\in \Lambda$ such that $V \cong \pihat_{\lambda,\nu} \cong \pihat_{w *
\lambda, w \nu}$ for all $w \in W$. In particular, $V$ has minimal type
$\overline{\nu}$ and infinitesimal character $\chi(\lambda, \nu - \lambda
- 2 \rho)$.
\end{theorem}

\subsection{The objects in a given isomorphism class}

The other aspect of classification is to identify the isomorphism
classes. In light of the above result, the task is to identify when
$\pihat_{\lambda,\nu} \cong \pihat_{\lambda',\nu'}$. In light of
Corollary \ref{Cequiv}, one may assume that $\nu = \nu' \in \Lambda^+$.
Moreover, in light of Corollaries \ref{Cequiv} and \ref{Czero} for
general $\lie{g}$, and Equation \eqref{Esimples} for $\lie{sl}_2(\C)$, it
is easy to guess the general result. This is further reinforced by the
fact that if $\nu,\nu' \in \Lambda$, and $\xi \in \lie{R}_\nu, \xi'
\in\lie{R}_{\nu'}$ (notation as in Theorem \ref{Thar}), then
$\pi_{\xi,\nu}$ is irreducible, hence isomorphic to
$\pihat_{\lambda,\nu}$ (and similarly for $\pi_{\xi',\nu'}$). Now as
observed in \cite{PRV2},
\begin{eqnarray*}
&& \pihat_{\lambda,\nu} \cong \pihat_{\lambda',\nu'} \Longleftrightarrow
\pi_{\xi,\nu} \cong \pi_{\xi',\nu'} \Longleftrightarrow \Theta_{\xi,\nu}
= \Theta_{\xi',\nu'}\\
& \Longleftrightarrow & \exists w \in W : (\xi',\nu') = (w \xi, w \nu)
\Longleftrightarrow \exists w \in W : (\lambda',\nu') = (w * \lambda, w
\nu).
\end{eqnarray*}

\noindent It should not come as a surprise now, that the obvious claim
turns out to be correct:

\begin{theorem}
Given $(\lambda, \nu), (\lambda', \nu') \in \lie{h}^* \times \Lambda$, \[
\pihat_{\lambda,\nu} \cong \pihat_{\lambda',\nu'} \Longleftrightarrow
\exists w \in W : (\lambda',\nu') = (w * \lambda, w \nu). \]
\end{theorem}

\subsection{Concluding remarks}

We end with a couple of (incomplete) calculations regarding the above
analysis, involving central characters.
\begin{enumerate}
\item It is natural to ask if the central character associated to an
irreducible admissible module $V$, determines its minimal type. Thus,
given that $\chi_V = \chi(\mu_1,\mu_2)$, how does one determine the
minimal type of $V$?

It is clear that if $V \cong \pihat_{\lambda,\nu}$ (from above results),
then
\[ \mu_1 = w_1 * \lambda, \qquad \mu_2 = w_2 * (\nu - \lambda - 2 \rho) =
w_2 \nu - w_2 * \lambda - 2 \rho, \]

\noindent for some $w_1, w_2 \in W$. Now note that
\[ \mu_1 + w_1 w_2^{-1} * \mu_2 = w_1 * \lambda + w_1 \nu - w_1 * \lambda
- 2 \rho = w_1 \nu - 2 \rho. \]

\noindent Hence $\nu_w := 2 \rho + \mu_1 + w * \mu_2 \in \Lambda$ for
some $w \in W$; moreover, for every such $w$, $\overline{\nu_w}$ is a
candidate for the minimal type, by these calculations. Thus, if $w$ is
not uniquely identified from above, then neither is $\overline{\nu}$.

\item Similarly, given $(\lambda, \nu), (\lambda', \nu') \in \lie{h}^*
\times \Lambda$, a necessary condition for $\pihat_{\lambda, \nu}$ to be
isomorphic to $\pihat_{\lambda',\nu'}$ is that their central characters
and minimal types coincide. It is natural to ask if this data is also
sufficient to determine the isomorphism type.

Clearly, in order to have the same minimal type, Corollary \ref{Cequiv}
implies that $\nu' \in W \nu$. Say $\nu' = w_1 \nu$. Now since the
infinitesimal characters coincide, Theorem \ref{Tchar} implies:
\[ \lambda' = w_2 * \lambda, \qquad \nu' - \lambda' - 2 \rho = w * (\nu -
\lambda - 2 \rho) = w \nu - w * \lambda - 2 \rho. \]

\noindent Using these equations translates to the following condition:
\[ w_1 \nu - w_2 * \lambda = w \nu - w * \lambda, \]

\noindent and this data may not have the unique solution: $w_1 = w_2 =
w$.

The reason for this discrepancy is Equation \eqref{Ehar}: the
representation $\pihat_{\lambda,\nu}$ carries the same data as its
minimal type and the action of $\Omega$ on it. The above data only
accounts for the minimal type and the action of the proper subset
$Z(\lie{U}(\lie{g} \times \lie{g})) \subsetneq \Omega$. For instance,
$Z(\lie{U} \overline{\lie{g}})$ is not accounted for.
\end{enumerate}\bigskip

To conclude, we have tried to explain the flavour of some of the results
in \cite{PRV2}, as well as their connection to, and impact on, subsequent
research in a wide variety of directions in the field. From the
multiplicity problem and obtaining components in tensor products of
finite-dimensional modules, to PRV determinants and annihilators of Verma
modules, to the classification of all irreducible admissible modules as
in Harish-Chandra's grand program on semisimple Lie groups - the work
\cite{PRV2} has contributed to, and inspired much subsequent research in,
many aspects of representation theory. The list of results and
connections mentioned in this article is by no means complete, but we
hope that it suffices to convince the reader of the importance and
influence of this work in representation theory.

\subsection*{Acknowledgments}

I would first like to thank Professors C.S.~Rajan and Rajendra Bhatia for
kindly inviting me to write this article. Next, I greatly thank Professor
Shrawan Kumar for his generous help in pointing out several references
and follow-up results in the literature. I would also like to thank very
much Professor Dipendra Prasad, for his time and patience in answering
several of my questions, as well as Professor Vyjayanthi Chari for useful
references and discussions. An excellent historical account of the
writing of \cite{PRV2} can be found in Professor V.S.~Varadarajan's
reminiscences \cite{Var}, and this was of great help in the writing of
the present work as well.


\begin{thebibliography}{MPR2}
\bibitem[BB]{BB}
A.A.~Beilinson and J.N.~Bernstein, {\em Localisation de
  $\lie{g}$-modules}, C.~R.~Acad.~Sci.~Paris, Ser.~1 \textbf{292} (1981),
  15--18.

\bibitem[BZ]{BZ}
A.~Berenstein and A.~Zelevinsky, {\em Tensor product multiplicities,
  canonical bases and totally positive varieties}, Inventiones
  Mathematicae \textbf{143} (2001), 77--128.

\bibitem[BGG]{BGG1}
J.~Bernstein, I.M.~Gelfand, and S.I.~Gelfand, \emph{A category of
  $\mathfrak{g}$ modules}, Functional Analysis and Applications
  \textbf{10} (1976), 87--92.

\bibitem[Br]{Br}
R.~Brauer, {\em Sur la multiplication des caract\'eristiques des groups
  continus et semi-simples}, C.~R.~Acad.~Sci. Paris \textbf{204} (1937),
  1784--1786.

\bibitem[Bru]{Bru}
F.~Bruhat, {\em Sur les r\'epr\'esentations induites des groupes de Lie},
  Bull.~Soc.~Math.~France \textbf{84} (1956), 97--205.

\bibitem[CG1]{CG1}
V.~Chari and J.~Greenstein, {\em Current algebras, highest weight
  categories and quivers}, Advances in Mathematics \textbf{216} (2007),
  no.~2, 811--840.

\bibitem[CG2]{CG3}
\bysame, {\em Minimal affinizations as projective objects}, Journal of
  Geometry and Physics \textbf{61} (2011), no.~3, 594--609.

\bibitem[CKR]{CKR}
V.~Chari, A.~Khare, and T.B.~Ridenour, {\em Faces of polytopes and Koszul
  algebras}, Journal of Pure and Applied Algebra \textbf{216} (2012),
  no.~7, 1611--1625.

\bibitem[CP]{CP}
V.~Chari and A.~Pressley, {\em A new family of irreducible,integrable
  modules for affine Lie algebras}, Mathematische Annalen \textbf{277}
  (1987), no.~3, 543--562.

\bibitem[DR]{DR}
I.~Dimitrov and M.~Roth, {\em Geometric realization of PRV components and
  the Littlewood-Richardson cone}, Contemporary Mathematics \textbf{490}:
  Symmetry in Mathematics and Physics, D.~Babbitt, V.~Chari, and
  R.~Fioresi, Eds. (2009), 83--95.

\bibitem[Du1]{Du1}
M.~Duflo, {\em Construction of primitive ideals in an enveloping
  algebra}, Lie groups and their representations (1971 J\'anos Bolyai
  Math.~Soc.~Summer School in Mathematics, Budapest), I.M.~Gelfand, Ed.
  (1975), 77--93.

\bibitem[Du2]{Du2}
\bysame, {\em Repr\'esentations irr\'eductibles des groupes semi-simples
  complexes}, Springer Lecture Notes in Mathematics \textbf{497} (1975),
  26--88.

\bibitem[Du3]{Du3}
\bysame, {\em Sur la classification des id\'eaux primitifs dans
  l'alg\`ebre enveloppante d'une alg\`ebre de Lie semi-semisimple},
  Annals of Mathematics \textbf{105} (1977), 107--120.

\bibitem[FL]{FL}
D.R.~Farkas and G.~Letzter, {\em Quantized representation theory
  following Joseph}, Progress in Mathematics \textbf{243}, Part I:
  Studies in Lie Theory (2006), 9--17.

\bibitem[GN]{GN}
I.M.~Gelfand and M.A.~Naimark, {\em Unitary representations of the
  classical groups}, Trudy Mat.~Inst.~Steklov \textbf{36},
  Moscow-Leningrad, 1950.

\bibitem[GL]{GL}
M.~Gorelik and E.~Lanzmann, {\em The annihilation theorem for Lie
  superalgebra $\mathfrak{osp}(1,2\ell)$}, Inventiones Mathematicae
  \textbf{137} (1999), 651--680.

\bibitem[Hal]{Ha}
B.C.~Hall, {\em Lie groups, Lie algebras, and representations: an
  elementary introduction}, Graduate Texts in Mathematics,
  no.~\textbf{222}, Springer-Verlag, Berlin-New York, 2004.

\bibitem[Har1]{Har1}
Harish-Chandra, {\em On some applications of the universal enveloping
  algebra of a semi-simple Lie algebra},
  Trans.~Amer.~Math.~Soc.~\textbf{70} (1951), 28--96.

\bibitem[Har2]{Har2}
\bysame, {\em Representations of a semi-simple Lie group on a Banach
  space: I}, Trans.~Amer.~Math.~Soc.~\textbf{75} (1953), 185--243.

\bibitem[Har3]{Har3}
\bysame, {\em Representations of semi-simple Lie groups: II},
  Trans.~Amer.~Math.~Soc.~\textbf{76} (1954), 26--65.

\bibitem[Har4]{Har4}
\bysame, {\em The Plancherel formula for complex semi-simple Lie groups},
  Trans.~Amer.~Math.~Soc.~\textbf{76} (1954), 485--528.

\bibitem[Hay]{Hay}
M.~Hayashi, {\em The moduli space of $SU(3)$-flat connections and the
  fusion rules}, Proc.~Amer.~Math.~Soc.~\textbf{127} (1999), 1545-1555.

\bibitem[Hu1]{H}
J.E.~Humphreys, \emph{Introduction to Lie algebras and representation
  theory}, Graduate Texts in Mathematics, no.~\textbf{9},
  Springer-Verlag, Berlin-New York, 1972.

\bibitem[Hu2]{H2}
\bysame, {\em Representations of semisimple Lie algebras in the BGG
  Category $\calo$}, Graduate Studies in Mathematics \textbf{94},
  American Mathematical Society, Providence, RI, 2008.

\bibitem[Jo1]{Jo1}
A.~Joseph, {\em Quantum groups and their primitive ideals},
  Ergeb.~Math.~Grenzgeb.~(3) \textbf{29}, Springer, Berlin, 1995.

\bibitem[Jo2]{Jo2}
\bysame, {\em A completion of the quantized enveloping algebra of a
  Kac-Moody algebra}, Journal of Algebra \textbf{214} (1999), no.~1,
  235--275.

\bibitem[Jo3]{Jo3}
\bysame, {\em On the Kostant-Parthasarathy-Ranga Rao-Varadarajan
  determinants, I. Injectivity and multiplicities}, Journal of Algebra
  \textbf{241} (2001), 27--45.

\bibitem[JL1]{JL1}
A.~Joseph and G.~Letzter, {\em Separation of variables for quantized
  enveloping algebras}, American Journal of Mathematics \textbf{116}
  (1994), 127--177.

\bibitem[JL2]{JL2}
\bysame, {\em Verma modules annihilators for quantized enveloping
  algebras}, Ann.~Ecole Norm.~Sup.~\textbf{28} (1995), 493--526.

\bibitem[JL3]{JL3}
\bysame, {\em On the Kostant-Parthasarathy-Ranga Rao-Varadarajan
  determinants, II. Construction of the KPRV determinants}, Journal of
  Algebra \textbf{241} (2001), 46--66.

\bibitem[JLT]{JLT}
A.~Joseph, G.~Letzter, and D.~Todoric, {\em On the
  Kostant-Parthasarathy-Ranga Rao-Varadarajan determinants, III.
  Computation of the KPRV determinants}, Journal of Algebra \textbf{241}
  (2001), 67--88.

\bibitem[JT]{JT}
A.~Joseph and D.~Todoric, {\em On the quantum KPRV determinants for
  semisimple and affine Lie algebras}, Algebras and Representation Theory
  \textbf{5} (2002), 57--99.

\bibitem[Ka]{Ka}
M.~Kashiwara, {\em On crystal bases}, Canadian
  Math.~Soc.~Conf.~Proc.~\textbf{16} (1995), 155-197.

\bibitem[KLV]{KLV}
D.~Kazhdan, M.~Larsen, and Y.~Varshavsky, {\em The Tannakian Formalism
  and the Langlands Conjectures}, preprint,
  \href{http://arxiv.org/abs/1006.3864}{\tt math.NT/1006.3864}.

\bibitem[Kh]{Kh3}
A.~Khare, \emph{Axiomatic framework for the BGG Category $\mathcal{O}$},
  preprint, \href{http://arxiv.org/abs/0811.2080}{\tt math.RT/0811.2080}.

\bibitem[Kl]{Kl}
A.U.~Klimyk, {\em On the multiplicities of weights of representations and
  the multiplicities of representations of semisimple Lie algebras},
  Dokl.~Acad.~Nauk SSSR \textbf{177} (1967), 1001--1004.

\bibitem[KZ]{KZ}
A.W.~Knapp and G.~Zuckerman, {\em Classification of irreducible tempered
  representations of semisimple groups}, Annals of Mathematics
  \textbf{116} (1982), 389--455.

\bibitem[Ko1]{Kos1}
B.~Kostant, {\em A formula for the multiplicity of a weight},
  Trans.~Amer.~Math.~Soc.~\textbf{93} (1959), 53--73.

\bibitem[Ko2]{Kos2}
\bysame, {\em Lie group representations on polynomial rings},
  American Journal of Mathematics \textbf{85} (1963), no.~3, 327--404.

\bibitem[Ko3]{Kos3}
\bysame, {\em On the existence and irreducibility of certain series of
  representations}, Lie groups and their representations (1971 J\'anos
  Bolyai Math.~Soc.~Summer School in Mathematics, Budapest),
  I.M.~Gelfand, Ed. (1975), 231--331.

\bibitem[Ko4]{Kos4}
\bysame, {\em Clifford algebra analogue of the Hopf-Koszul-Samelson
  Theorem, the $\rho$-decomposition $C(\mathfrak{g}) = \End V_\rho
  \otimes C(P)$, and the $\mathfrak{g}$-module structure of $\wedge
  \mathfrak{g}$}, Advances in Mathematics \textbf{125} (1997), 275--350.

\bibitem[Ku1]{Kum1}
S.~Kumar, {\em Proof of the Parthasarathy-Ranga Rao-Varadarajan
  conjecture}, Inventiones Mathematicae \textbf{93} (1988), 117--130.

\bibitem[Ku2]{Kum2}
\bysame, {\em Existence of certain components in the tensor product of
  two integrable highest weight modules for Kac-Moody algebras}, Advanced
  series in Mathematical Physics \textbf{7}: Infinite dimensional Lie
  algebras and groups, V.G.~Kac, Ed. (1989), 25--38.

\bibitem[Ku3]{Kum3}
\bysame, {\em A refinement of the PRV conjecture}, Inventiones
  Mathematicae \textbf{97} (1989), 305--311.

\bibitem[Ku4]{Kum4}
\bysame, {\em Proof of Wahl's conjecture on surjectivity of the Gaussian
  map for flag varieties}, American Journal of Mathematics \textbf{114}
  (1992), 1201--1220.

\bibitem[Ku5]{Kum5}
\bysame, {\em Tensor Product Decomposition}, Proceedings of the
  International Congress of Mathematicians (2010).

\bibitem[La]{La}
R.P.~Langlands, {\em On the classification of irreducible representations
  of real algebraic groups}, Mathematical Surveys and Monographs
  \textbf{31}: Representation theory and harmonic analysis on semisimple
  Lie groups, P.J.~Sally and D.A.~Vogan, Eds. (1989), 101--170.

\bibitem[Le]{Le}
J.~Lepowsky, {\em Algebraic results on representations of semisimple Lie
  groups}, Trans.~Amer.~Math.~Soc.~\textbf{176} (1973), 1--44.

\bibitem[LMC]{LMC}
J.~Lepowsky and G.W.~McCollum, {\em On the determination of irreducible
  modules by restriction to a subalgebra},
  Trans.~Amer.~Math.~Soc.~\textbf{176} (1973), 45--57.

\bibitem[Li]{Li}
P.~Littelmann, {\em A Littlewood-Richardson rule for symmetrizable
  Kac-Moody algebras}, Inventiones Mathematicae \textbf{116} (1994),
  329--346.

\bibitem[Lu]{Lu}
G.~Lusztig, {\em Canonical bases arising from quantized enveloping
  algebras II}, Prog.~Theor.~Phys.~\textbf{102} (1990), 175--201.

\bibitem[Ma1]{Mat1}
O.~Mathieu, {\em Construction d'un groupe de Kac-Moody et applications},
  Compositio Mathematica \textbf{69} (1989), no.~1, 37--60.

\bibitem[Ma2]{Mat2}
\bysame, {\em Classification of Harish-Chandra modules over the Virasoro
  Lie algebra}, Inventiones Mathematicae \textbf{107} (1992), 225--234.

\bibitem[Ma3]{Mat3}
\bysame, {\em Classification of irreducible weight modules}, Annales de
  l'institut Fourier \textbf{50} (2000), no.~2, 537--592.

\bibitem[MPR1]{MPR1}
P.L.~Montagard, B.~Pasquier, and N.~Ressayre, {\em Two generalizations of
  the PRV conjecture}, Composition Mathematica \textbf{147} (2011),
  no.~4, 1321--1336.

\bibitem[MPR2]{MPR2}
\bysame, {\em Generalizations of the PRV conjecture, II}, preprint,
  \href{http://arxiv.org/abs/1110.4621}{\tt math.RT/1110.4621}.

\bibitem[MP]{MP}
R.V.~Moody and A.~Pianzola, \emph{Lie algebras with triangular
  decompositions}, Canadian Mathematical Society Series of Monographs and
  Advanced Texts, Wiley Interscience, New York-Toronto, 1995.

\bibitem[PY]{PY}
D.I.~Panyushev and O.S.~Yakimova, {\em The PRV-formula for tensor product
  decompositions and its applications}, Functional Analysis and
  Applications \textbf{42} (2008), no.~1, 45--52.

\bibitem[PRV1]{PRV1}
K.R.~Parthasarathy, R.~Ranga Rao, and V.S.~Varadarajan, {\em
  Representations of complex semisimple Lie groups and Lie algebras},
  Bull.~Amer.~Math.~Soc.~\textbf{72} (1966), 522--525.

\bibitem[PRV2]{PRV2}
\bysame, {\em Representations of complex semisimple Lie groups and Lie
  algebras}, Annals of Mathematics \textbf{85} (1967), 383--429.

\bibitem[Po]{Po}
P.~Polo, {\em Vari\'et\'es de Schubert et excellentes filtrations},
  Ast\'erisque (Orbites unipotentes et repr\'esentations)
  \textbf{173-174} (1989), 281--311.

\bibitem[Ra]{Ra}
K.N.~Rajeswari, {\em Standard monomial theoretic proof of PRV
  conjecture}, Communications in Algebra \textbf{19} (1991), 347--425.

\bibitem[Sh]{Sha}
N.N.~Shapovalov, \emph{On a bilinear form on the universal enveloping
  algebra of a complex semisimple Lie algebra}, Functional Analysis and
  Its Applications \textbf{6} (1972), 307--312.

\bibitem[St]{St}
R.~Steinberg, {\em A general Clebsch-Gordan Theorem},
  Bull.~Amer.~Math.~Soc.~\textbf{67} (1961), 406--407.

\bibitem[Va]{Var}
V.S.~Varadarajan, {\em Some mathematical reminiscences}, Methods and
  Applications of Analysis \textbf{9} (2002), no.~3, v--xviii.

\bibitem[Ve]{Ver}
D.N.~Verma, {\em Structure of certain induced representations of complex
  semisimple Lie algebras}, Bull.~Amer. Math.~Soc.~\textbf{74} (1968),
  no.~1, 160--166.

\bibitem[Vo]{Vo}
D.A.~Vogan Jr., {\em The algebraic structure of the representation of
  semisimple Lie groups.~I}, Annals of Mathematics \textbf{109} (1979),
  no.~1, 1--60.

\bibitem[YZ]{YZ}
C.A.S.~Young and R.~Zegers, {\em Dorey's rule and the q-characters of
  simply-laced quantum affine algebras}, Commun.~Math.~Phys.~\textbf{302}
  (2011), 789--813.

\bibitem[Zh1]{Zh1}
D.P.~Zhelobenko, {\em The analysis of irreducibility in the class of
  elementary representations of a complex semisimple Lie group},
  Math.~USSR Izv.~\textbf{2} (1968), no.~1, 105--128.

\bibitem[Zh2]{Zh2}
\bysame, {\em Harmonic analysis on complex semisimple Lie groups}
  (Russian), Mir, Moscow, 1974.
\end{thebibliography}
\end{document}